\documentclass{amsart}

\usepackage{amssymb}
\usepackage{amsfonts}
\usepackage{amsmath}
\usepackage{amsthm}
\usepackage{graphicx}
\usepackage{mathrsfs}
\usepackage{dsfont}
\usepackage{amscd}
\usepackage{multirow}
\usepackage{palatino}
\usepackage{mathpazo}
\usepackage[all]{xy}
\usepackage[scaled]{helvet} % ss
\usepackage{courier} % tt
\normalfont
\usepackage[T1]{fontenc}
\usepackage{calligra}
\usepackage{verbatim}
\usepackage[usenames,dvipsnames]{color}
\usepackage[colorlinks=true,linkcolor=Blue,citecolor=Violet]{hyperref}

\makeatletter
\@namedef{subjclassname@2010}{%
  \textup{2010} Mathematics Subject Classification}
\makeatother

\theoremstyle{plain}
\newtheorem{theorem}{Theorem}[section]

\newtheorem{corollary}[theorem]{Corollary}

\newtheorem{lemma}[theorem]{Lemma}
\newtheorem{proposition}[theorem]{Proposition}

\newtheorem{theorem-definition}[theorem]{Theorem-Definition}

\theoremstyle{definition}
\newtheorem{definition}[theorem]{Definition}

\newtheorem{notation}[theorem]{Notation}

\theoremstyle{remark}

\newtheorem{remark}[theorem]{Remark}

\numberwithin{equation}{section}

\newcommand{\ee}{{\varepsilon}}
\newcommand{\dd}{{\mathsf{d}}}
\newcommand{\N}{{\mathds{N}}}

\newcommand{\R}{{\mathds{R}}}
\newcommand{\C}{{\mathds{C}}}

\newcommand{\D}{{\mathfrak{D}}}
\newcommand{\A}{{\mathfrak{A}}}
\newcommand{\B}{{\mathfrak{B}}}

\newcommand{\Lip}{{\mathsf{L}}}

\newcommand{\dist}{{\mathsf{dist}}}

\newcommand{\qpropinquity}[1]{{\mathsf{\Lambda}_{#1}}}

\newcommand{\Kantorovich}[1]{{\mathsf{mk}_{#1}}}

\newcommand{\Haus}[1]{{\mathsf{Haus}_{#1}}}

\newcommand{\StateSpace}{{\mathscr{S}}}

\newcommand{\mongekant}{{Mon\-ge-Kan\-to\-ro\-vich metric}}

\newcommand{\unit}{1}

\newcommand{\sa}[1]{{\mathfrak{sa}\left({#1}\right)}}

\newcommand{\dom}[1]{{\operatorname*{dom}\left({#1}\right)}}
\newcommand{\diam}[2]{{\mathrm{diam}\left({#1},{#2}\right)}}

\newcommand{\bridgereach}[2]{{\varrho\left({#1}\middle|{#2}\right)}}
\newcommand{\bridgeheight}[2]{{\varsigma\left({#1}\middle|{#2}\right)}}

\newcommand{\bridgelength}[2]{{\lambda\left({#1}\middle|{#2}\right)}}

\newcommand{\bridgenorm}[2]{{\mathsf{bn}_{ {#1}  }\left({#2}\right)}}

\newcommand{\alg}[1]{{\mathfrak{#1}}}

\renewcommand{\geq}{\geqslant}
\renewcommand{\leq}{\leqslant}

\newcommand{\Latremoliere}{Latr\'{e}moli\`{e}re}

\allowdisplaybreaks[1]

\makeatletter
\newcommand{\vast}{\bBigg@{4}}
\newcommand{\Vast}{\bBigg@{5}}
\makeatother

\begin{document}

\title[Quantum metrics on tensor products of  commutative by AF algebras]{Quantum metrics on the tensor product of a commutative C*-algebra and an AF C*-algebra}
\author{Konrad Aguilar}
\address{School of Mathematical and Statistical Sciences \\ Arizona State University \\  901 S. Palm Walk, Tempe, AZ 85287-1804}
\email{konrad.aguilar@asu.edu}
\urladdr{https://math.la.asu.edu/~kaguilar/}

\date{\today}
%%%%%%%%%%%%%%%%%%%%%%%%%%%%%%%%%%%%%%
\thanks{The author  was partially supported by  the grant H2020-MSCA-RISE-2015-691246-QUANTUM DYNAMICS and the Polish  Ministry of Science and Higher Education  grant  \#3542/H2020/2016/2.}

\begin{abstract}
Given a compact metric space $X$ and a unital AF algebra $A$ equipped with a faithful tracial state, we place quantum metrics   on the tensor product of $C(X)$ and $A$ given established quantum metrics on $C(X)$ and $A$  from work with Bice and \Latremoliere. We prove the inductive limit of $C(X)$ tensor $A$ given by     $A$ is a metric limit in the Gromov-Hausdorff propinquity. We  show that our quantum metric is compatible with the tensor product by producing a Leibniz rule on elementary tensors and showing the diameter of our quantum metric on the tensor product is bounded above the diameter of the Cartesian product of  the quantum metric spaces. We provide  continuous families of $C(X)$ tensor $A$ which extends our previous results with \Latremoliere{} on UHF algebras.  
\end{abstract}
\maketitle

\setcounter{tocdepth}{1}
\tableofcontents
\section{Introduction}

Compact quantum metric spaces introduced by Rieffel \cite{Rieffel98a, Rieffel05} and motivated by work of Connes \cite{Connes, Connes89} were developed to study the metric aspect of Noncommutative Geometry. In particular,   Rieffel developed the first noncommutative analogue of the Gromov-Hausdorff distance \cite{Rieffel00}. This allows one to establish continuous families of C*-algebras built from natural parameter spaces including the noncommutative tori \cite{Rieffel00, Latremoliere13c}, AF algebras \cite{Aguilar-Latremoliere15, Aguilar16}, certain C*-dynamical systems \cite{Kaad19}, etc. Many more quantum Gromov-Hausdorff distances followed \cite{Kerr02, Kerr09, Latremoliere13, Latremoliere13b} to only name a few.  However, in this article, we focus on the Gromov-Hausdorff propinquity of \Latremoliere{} \cite{Latremoliere13, Latremoliere13b}. The reason for this is that his quantum distances are built only with the category of C*-algebras in mind, which has also allowed him to introduce quantum distances for classes of Hilbert C*-modules \cite{Latremoliere16, Latremoliere18a} and for Spectral Triples \cite{Latremoliere18c}. In a similar manner, Rieffel was able to capitalize on these particular properties of \Latremoliere's propinquity to obtain his recent results in \cite{Rieffel15, Rieffel18}.

Our work thus far has been mostly  focused on using the tools of Noncommutative Metric Geometry to translate the categorical notion of a limit of spaces, an inductive limit, to a metric limit using the Gromov-Hausdorff propinquity, which   we have done for the class of all unital  AF algebras \cite{Aguilar-Latremoliere15, Aguilar16, Aguilar18}. However, in this article, our focus is on placing quantum metrics on certain tensor products of C*-algebras that recovers some of the structure of given quantum metrics on each C*-algebra that forms the tensor products. This is motivated simply by the fact that forming tensor products is often a useful tool for studying spaces in any given category, and thus, we hope our work begins to enrich the study compact quantum metric spaces in this manner as well. For instance, our work in this article provides Lip-norms that satisfy a Leibniz rule on elementary tensors using the original Lip-norms on each quantum metric space (Expression \eqref{eq:leibniz-tensor}). Furthermore, we show that we are able to bound the diameter of our quantum metric on the tensor product by the diameter of the Cartesian product of the quantum metrics, which reflects the classical structure of tensor products given by $C(X) \otimes C(Y) \cong C(X \times Y)$ (Expression \eqref{eq:tensor-diam} and Corollary \ref{c:comm-tensor-diam}).   The way we accomplish these results is that we use any inductive sequence that forms the AF algebra $\A$ as its inductive limit  to represent $C(X) \otimes \A$ as an inductive limit and extend our work with Bice in \cite{Aguilar-Bice17} on certain homogeneous C*-algebras and with \Latremoliere{} in \cite{Aguilar-Latremoliere15} on certain AF algebras. Hence, in this paper, we also get results that translate a categorical limit to a metric limit in propinquity and we show that this process is not affected by taking a tensor product with $C(X)$ if a suitable quantum metric is constructed as seen in Theorem \ref{t:ah-c*-cqms}. We note that all of the inductive limits in this article are    approximately homogeneous C*-algebras or AH algebras. Furthermore, in \cite[Section 4]{Aguilar-Latremoliere15}, it was shown that the class of UHF algebras form a continuous image of the Baire space using their multiplicity sequences  with respect to quantum propinquity via a $2$-Lipschitz map.  In this paper in Theorem \ref{t:uhf-tensor-cont}, we show that this result is unaffected by taking a tensor product with $C(X)$ including the $2$-Lipschitz property. Finally, in Theorem \ref{t:fd-approx}, we approximate distances between $C(X)\otimes \A$ and $C(Y)\otimes \A$ using structure from the Lip-norms on $C(X), C(Y)$ and $\A$. This, in turn, allows us to provide finite-dimensional approximations for $C(X)\otimes \A$ using finite-dimensional commutative C*-algebras. Hence, we believe our work provides a foundational and non-trivial example for quantum metrics on the tensor products of quantum metrics spaces that are compatible with the tensor product structure.

\section{Background}

This section provides some background for the results of this paper.  These results are taken from \cite{Rieffel00}, \cite{Latremoliere13}, and \cite{Aguilar-Latremoliere15},  which pertain to quantum metrics and quantum Gromov-Hausdorff distance, Gromov-Hausdorff propinquity, and quantum metrics on AF algebras, respectively.

We begin with results and definitions related to quantum metric spaces. The following definition of a compact quantum metric space is not the original definition given by order unit spaces, but we use the same terminology since it is standard in the literature on quantum metric spaces that the order unit spaces being considered are the self-adjoint elements of a unital C*-algebra. Also, we include in this definition the notion of quasi-Leibniz, which generalizes the notion of Leibniz to the noncommutative setting and was used by \Latremoliere \ to introduce his noncommutative analogue to the Gromov-Hausdorff distance called the quantum Gromov-Hausdorff propinquity \cite{Latremoliere13,Latremoliere15}.

\begin{notation}
Let $\A$ be a unital C*-algebra. The unit of $\A$ will be denoted by $\unit_\A$. The state space of $\A$ will be denoted by $\StateSpace(\A)$ while the self-adjoint part of $\A$ will be denoted by $\sa{\A}$.  The C*-norm of $\A$ will be denoted by $\|\cdot\|_\A$. 
\end{notation}

\begin{definition}[\cite{Rieffel98a,Rieffel99, Rieffel05,Latremoliere15}]\label{Monge-Kantorovich-def}
A {\em compact quantum  metric space} $(\A,\Lip)$ is an ordered pair where $\A$ is a unital C*-algebra    and $\Lip$ is a seminorm defined on a unital dense subspace $\dom{\Lip}$ of $\sa{\A}$   such that:
\begin{enumerate}
\item $\{ a \in \sa{\A} : \Lip(a) = 0 \} = \R\unit_\A$,
\item the \emph{\mongekant} defined, for all two states $\varphi, \psi \in \StateSpace(\A)$, by
\begin{equation*}
\Kantorovich{\Lip} (\varphi, \psi) = \sup\left\{ |\varphi(a) - \psi(a)| : a\in\dom{\Lip}, \Lip(a) \leq 1 \right\}
\end{equation*}
metrizes the weak* topology of $\StateSpace(\A)$, and
\item the seminorm $\Lip$ is lower semi-continuous on $\sa{\A}$ with respect to $\|\cdot\|_\A$.
\end{enumerate}
If $(\A,\Lip)$ is a  compact quantum metric space, then we call the seminorm $\Lip$ a {\em Lip-norm}.  

Furthermore, if there exists $C\geq 1$ such that 
\[
\Lip\left(\frac{ab+ba}{2}\right) \leq C (\Lip(a) \|b\|_\A+ \Lip(b)\|a\|_\A)
\]
and 
\[
\Lip\left(\frac{ab-ba}{2i}\right) \leq C (\Lip(a) \|b\|_\A+ \Lip(b)\|a\|_\A)
\]
for all $a,b \in \sa{\A}$, then we call $\Lip$, {\em $C$-quasi-Leibniz} and we call $(\A, \Lip)$ a { \em $C$-quasi-Leibniz compact quantum metric space.} 
\end{definition}

The following gathers most of the known characterizations for compact quantum metric spaces, which we will use in this article since it is often difficult to directly show that the Monge-Kantorovich metric metrizes the weak* topology. 

\begin{theorem}[\cite{Rieffel98a,Rieffel99, Ozawa05}]\label{Rieffel-thm}
Let $(\A,\Lip)$ be an ordered pair where $\A$ is unital C*-algebra and $\Lip$ is a lower semi-continuous seminorm defined on unital dense subspace of  $\sa{\A}$. The following are equivalent:
\begin{enumerate}
\item $(\A,\Lip)$ is a  compact quantum metric space; 
\item the metric $\Kantorovich{\Lip}$ is bounded and there exists $r \in \R, r >0$ such that the set:
\begin{equation*}
\{a \in \dom{\Lip} : \Lip(a) \leq 1 \text{ and } \Vert a \Vert_\A \leq r \}
\end{equation*}
is compact in $\A$ for $\Vert \cdot \Vert_\A $;
\item the set:
\begin{equation*} \{ a+\R1_\A \in \sa{\A}/\R1_\A : a \in \dom{\Lip}, \Lip(a) \leq 1 \}
\end{equation*} is compact in $\sa{\A}/\R1_\A$ for $\Vert \cdot \Vert_{\sa{\A}/\R1_\A}$; 
\item there exists a state $\mu \in \StateSpace (\A)$ such that the set: 
\begin{equation*}
\{ a\in \dom{\Lip} : \Lip (a) \leq 1 \text{ and } \mu(a) = 0 \}
\end{equation*}
is compact in $\A$ for $\Vert \cdot \Vert_\A $;
\item for all $\mu \in \StateSpace (\A)$ the set: 
\begin{equation*}
\{ a\in \dom{\Lip} : \Lip (a) \leq 1 \text{ and } \mu(a) = 0 \}
\end{equation*}
is compact in $\A$ for $\Vert \cdot \Vert_\A $.
\end{enumerate}
\end{theorem}

One of the main contributions of Noncommutative Metric Geometry are  generalizations of the Gromov-Hausdorff distance \cite{burago01} on the class of compact quantum metric spaces. There are several to choose from, but for our work in this article, we will use \Latremoliere's   quantum Gromov-Hausdorff propinquity of \cite{Latremoliere13}. The first noncommutative analogue to the Gromov-Hausdorff distance was introduce by Rieffel in \cite{Rieffel00}, and is known as the quantum Gromov-Hausdorff distance. The choice to use \Latremoliere's propinquity comes from the fact that it acknowledges the C*-algebraic structure on quasi-Leibniz compact quantum metric spaces, while also providing estimates using the notion of bridges which are a useful and powerful method for providing estimates. However, we note that \Latremoliere's propinquity dominates Rieffel's quantum distance \cite[Theorem 6.2]{Latremoliere13}  and thus all convergence results in this paper are valid for Rieffel's quantum distance as well. Furthermore, we note that \Latremoliere \ has another propinquity called the dual Gromov-Hausdorff propinquity, which is also complete on the certain classes of quasi-Leibniz spaces which comprise all spaces studied in this paper.  Also, the dual propinquity is dominated by the quantum propinquity \cite[Theorem 5.5]{Latremoliere13b}, and thus all convergence results in this paper are valid for dual propinquity as well.  Before we state the theorem for \Latremoliere's quantum   propinquity, we introduce some definitions from \cite{Latremoliere13}.

\begin{definition}[{\cite[Definition 3.1]{Latremoliere13}}]
The \emph{$1$-level set} $\StateSpace_1(\D|\omega)$ of an element $\omega$ of a unital C*-algebra $\D$ is
\begin{equation*}
\left\{ \varphi \in \StateSpace(\D) : \varphi((1-\omega^\ast\omega))=\varphi((1-\omega \omega^\ast)) = 0 \right\}\text{.}
\end{equation*}
\end{definition}
Next, we define the notion of a  \Latremoliere \ bridge, which is not only crucial in the definition of the quantum propinquity but also the convergence results of \Latremoliere \ in \cite{Latremoliere13c} and Rieffel in \cite{Rieffel15}. In particular, the pivot of Definition (\ref{bridge-def}) is of utmost importance in the convergence results of \cite{Latremoliere13c, Rieffel15}.
\begin{definition}[{\cite[Definition 3.6]{Latremoliere13}}]\label{bridge-def}
A \emph{bridge} from $\A$ to $\B$, where $\A$ and $\B$ are unital C*-algebras, is a quadruple $(\D,\pi_\A,\pi_\B,\omega)$ where
\begin{enumerate}
\item $\D$ is a unital C*-algebra,
\item the element $\omega \in \D$, called the \emph{pivot} of the bridge, satisfies   $\StateSpace_1(\D|\omega) \not=\emptyset$,
\item $\pi_\A : \A\hookrightarrow \D$ and $\pi_\B : \B\hookrightarrow\D$ are unital *-monomorphisms.
\end{enumerate}
\end{definition}

In the next few definitions, we denote by $\Haus{\mathrm{d}}$ the \emph{Hausdorff (pseudo)distance} induced by a (pseudo)distance $\mathrm{d}$ on the compact subsets of a (pseudo)metric space $(X,\mathrm{d})$ \cite{Hausdorff}.

\begin{definition}[{\cite[Definition 3.16]{Latremoliere13}}]\label{d:height}
Let $C \geq 1$. Let $(\A,\Lip_\A)$ and $(\B,\Lip_\B)$ be two $C$-quasi-Leibniz compact quantum  metric spaces. The \emph{height} $\bridgeheight{\gamma}{\Lip_\A,\Lip_\B}$ of a bridge \\ $\gamma = (\D,\pi_\A,\pi_\B,\omega)$ from $\A$ to $\B$, and with respect to $\Lip_\A$ and $\Lip_\B$, is given by
\begin{equation*}
\max\left\{ \Haus{\Kantorovich{\Lip_\A}}(\StateSpace(\A), \pi_\A^\ast(\StateSpace_1(\D|\omega))), \Haus{\Kantorovich{\Lip_\B}}(\StateSpace(\B), \pi_\B^\ast(\StateSpace_1(\D|\omega))) \right\}\text{,}
\end{equation*}
where $\pi_\A^{\ast}$ and $\pi_\B^\ast$ are the dual maps of $\pi_\A$ and $\pi_\B$, respectively.
\end{definition}

\begin{definition}[{\cite[Definition 3.10]{Latremoliere13}}]
Let $C \geq 1$. Let $(\A,\Lip_\A)$ and $(\B,\Lip_\B)$ be two $C$-quasi-Leibniz compact quantum  metric spaces. The \emph{bridge seminorm} $\bridgenorm{\gamma}{\cdot}$ of a bridge $\gamma = (\D,\pi_\A,\pi_\B,\omega)$ from $\A$ to $\B$ is the seminorm defined on $\A\oplus\B$ by
\begin{equation*}
\bridgenorm{\gamma}{a,b} = \|\pi_\A(a)\omega - \omega\pi_\B(b)\|_\D
\end{equation*}
for all $(a,b) \in \A\oplus\B$.
\end{definition}

We implicitly identify $\A$ with $\A\oplus\{0\}$ and $\B$ with $\{0\}\oplus\B$ in $\A\oplus\B$ in the next definition, for any two spaces $\A$ and $\B$.

\begin{definition}[{\cite[Definition 3.14]{Latremoliere13}}]\label{d:reach}
Let $C \geq 1$. Let $(\A,\Lip_\A)$ and $(\B,\Lip_\B)$ be two $C$-quasi-Leibniz compact quantum  metric spaces. The \emph{reach} $\bridgereach{\gamma}{\Lip_\A,\Lip_\B}$ of a bridge \\
 $\gamma = (\D,\pi_\A,\pi_\B,\omega)$ from $\A$ to $\B$, and with respect to $\Lip_\A$ and $\Lip_\B$, is given by
\begin{equation*}
\Haus{\bridgenorm{\gamma}{\cdot}}\left( \left\{a\in\sa{\A} : \Lip_\A(a)\leq 1\right\} , \left\{ b\in\sa{\B} : \Lip_\B(b) \leq 1 \right\}  \right) \text{.}
\end{equation*}
\end{definition}

The next quantity is  a natural way to combine the information given by the height and the reach of a bridge.

\begin{definition}[{\cite[Definition 3.17]{Latremoliere13}}]\label{d:length}
Let $C \geq 1$. Let $(\A,\Lip_\A)$ and $(\B,\Lip_\B)$ be two $C$-quasi-Leibniz compact quantum  metric spaces. The \emph{length} $\bridgelength{\gamma}{\Lip_\A,\Lip_\B}$ of a bridge $\gamma = (\D,\pi_\A,\pi_\B,\omega)$ from $\A$ to $\B$, and with respect to $\Lip_\A$ and $\Lip_\B$, is given by
\begin{equation*}
\max\left\{\bridgeheight{\gamma}{\Lip_\A,\Lip_\B}, \bridgereach{\gamma}{\Lip_\A,\Lip_\B}\right\}\text{.}
\end{equation*}
\end{definition}

Although we will not define \Latremoliere's quantum  Gromov-Hausdorff propinquity, the following result provides the main tool for which we will furnish upper bounds for the quantum Gromov-Hausdorff propinquity.  This will then produce our approximation and convergence results. 

\begin{theorem}[{\cite[Theorem 6.1]{Latremoliere13}, \cite{Latremoliere15}}]\label{t:distq}
Let $C\geq 1$. The quantum Gromov-Hausdorff propinquity $\qpropinquity{}$ is a   metric on the full quantum isometry classes of $C$-quasi-Leibniz compact quantum metric spaces, where full quantum   isometry is given by a *-isomorphism $\pi: \A \rightarrow \B$ such that $\Lip_\B=\Lip_\A \circ \pi$, and thus
\[\qpropinquity{}((\A, \Lip_\A), (\B, \Lip_\B))=0\] 
for two $C$-quasi-Leibniz compact quantum metric spaces $(\A, \Lip_\A), (\B, \Lip_\B)$ if and only if $(\A, \Lip_\A), (\B, \Lip_\B)$ are fully quantum isometric.

Furthermore, if $\gamma$ is a bridge from $(\A, \Lip_\A)$ to $ (\B, \Lip_\B)$, then 
\[
\qpropinquity{}((\A, \Lip_\A), (\B, \Lip_\B))\leq  \bridgelength{\gamma}{\Lip_\A,\Lip_\B}.
\]
\end{theorem}
Next, the main goal of this article is to extend our results on the   homogeneous C*-algebras of the form $C(X, \A)$ for $X$ compact metric and $\A$ finite-dimensional in \cite{Aguilar-Bice17} to the case of $\A$ as a unital AF algebra equipped with a faithful tracial state.  These cover a part of a class of C*-algebras known as {\em approximately homogeneous} \cite[Definition V.2.1.9]{Blackadar06}.  Our approach to place quantum metrics on $C(X, \A)$ for an AF-algebra $\A$ is to use the quantum metrics on  homogeneous C*-algebras of the form $C(X, \B)$ with $\dim(\B)< \infty$ from \cite{Aguilar-Bice17} combined with the techniques for AF-algebras used in \cite{Aguilar-Latremoliere15}  along with techniques from \cite{Latremoliere05, Kerr09}, and we will make mention of this when these techniques appear in the rest of the article.  To accomplish this, we actually need to  create new quantum metrics on $C(X, \B)$ with $\dim(\B)< \infty$ that use much of the work in \cite{Aguilar-Bice17}.   These new quantum metrics will be introduced in the next section. 

Part of our goal is to show that our work in this article extends the AF-algebra case in \cite{Aguilar-Latremoliere15} by way of $C(\{x\}, \A)\cong \A$, and thus, we finish this background section, by stating this result.

\begin{theorem}[{\cite[Theorem 3.5]{Aguilar-Latremoliere15}}]\label{AF-lip-norms-thm-union} 
Let $\A$ be a unital AF algebra  with unit $1_\A$ endowed with a faithful tracial state $\mu$. Let $\mathcal{I} = (\A_n)_{n\in\N}$ be an increasing sequence of unital finite dimensional C*-subalgebras such that $\A=\overline{\cup_{n \in \N} \A_n}^{\Vert \cdot \Vert_\A}$ with $\A_0=\C 1_\A $.

Let $\pi$ be the GNS representation of $\A$ constructed from $\mu$ on the space $L^2(\A,\mu)$.

For all $n\in\N$, let:
\begin{equation*}
E_{\tau,n} : \A\rightarrow\A_n
\end{equation*}
be the unique $\mu$-preserving conditional expectation of $\A$ onto $\A_n$ induced by $L^2(\A,\mu)$.

Let $\beta: \N\rightarrow (0,\infty)$ have limit $0$ at infinity. If, for all $a\in\sa{\cup_{n \in \N} \A_n }$, we set
\begin{equation*}
\Lip_{\mathcal{I},\mu}^\beta(a) = \sup\left\{\frac{\left\|a - E_{\tau,n}(a)\right\|_\A}{\beta(n)} : n \in \N \right\},
\end{equation*}
then $\left(\A,\Lip_{\mathcal{I},\mu}^\beta\right)$ is a $2$-quasi-Leibniz compact quantum metric space. Moreover, for all $n\in\N$
\begin{equation*}
\qpropinquity{} \left(\left(\A_n,\Lip_{\mathcal{I},\mu}^\beta \right), \left(\A,\Lip_{\mathcal{I},\mu}^\beta \right)\right) \leq \beta(n)
\end{equation*}
and thus
\begin{equation*}
\lim_{n\rightarrow\infty} \qpropinquity{}\left(\left(\A_n,\Lip_{\mathcal{I},\mu}^\beta \right), \left(\A,\Lip_{\mathcal{I},\mu}^\beta\right)\right) = 0\text{.}
\end{equation*}
\end{theorem}

\section{Quantum metrics on $C(X)\otimes \A$}
To place quantum metrics on $C(X)\otimes \A$ for $X$ a compact metric and $\A$ a unital AF algebra equipped with a faithful tracial state, first, we will view $C(X)\otimes \A$ as $C(X, \A)$ and then extend the our work with Bice in \cite{Aguilar-Bice17}, where we placed quantum metrics on certain homogeneous C*-algebras, and extend our work with \Latremoliere{} in \cite{Aguilar-Latremoliere15}, where we place quantum metrics on certain AF algebras. We will show our construction recovers the original quantum metrics in many natural ways  including showing a Leibniz rule for elementary tensors in Theorem \ref{t:ah-c*-cqms}. Now, we establish some results about $C(X, \A)$. 

Given a compact metric space $(X, \mathsf{d}_X)$ and a unital AF-algebra $\A=\overline{\cup_{n \in \N} \A_n}^{\|\cdot\|_\A}$ with $\A_0\subseteq \A_1\subseteq \A_2 \subseteq \cdots$ with $\dim(\A_n)< \infty$ for all $n \in \N$, the C*-algebra $C(X, \A)$ is an inductive limit of $(C(X, \A_n))_{n \in \N}$. Now, when $\A$ is equipped with a faithful tracial state, we showed in \cite{Aguilar-Latremoliere15} that $\A$ and each $\A_n$ can be equipped with Lip-norms for which $\A$ is the metric limit of $\A_n$ in \Latremoliere's propinquity.  Hence, our goal is to do the same for  $C(X, \A)$ and $(C(X, \A_n))_{n \in \N}$ in this case. 

 First, we establish that $C(X, \A)$ is an inductive limit in the obvious way if $\A$ is an inductive limit.  This is a classic result that may be difficult to  find in the literature, and thus provide a proof here. We  utilize the argument outlined in \cite[Theorem 3.4]{Kaplansky51} to obtain the following.

\begin{proposition}\label{p:cx-dense}
Let $(X, \dd_X)$ be a compact metric space. Let $\A=\overline{\cup_{n\in \N} \A_n}^{\|\cdot\|_\A}$ be a unital  C*-algebra such that  $\A_n $ is a   C*-subalgebra of $\A$ containing $1_\A$ and $\A_n \subseteq \A_{n+1}$  for each   $n \in \N$.

Consider the unital C*-algebra 
\[
C(X, \A)=\{ f: X \rightarrow \A \mid f \text{ is continuous}\}
\]
equipped with point-wise operations and supremum norm induced by $\A$, and unit defined for all $x \in X$ by $1_{C(X, \A)}(x)=1_\A$.  Note that
$
C(X, \A_n)$  is a unital C*-subalgebra of $C(X,\A)$.  Furthermore, 
\[
C(X, \A)= \overline{\cup_{n \in \N} C(X, \A_n)}^{\|\cdot\|_{C(X, \A)}}.
\]
In particular, if $\A$ is AF, then $C(X, \A)$ is AH (approximately homogeneous \cite[Definition V.2.1.9]{Blackadar06}).
\end{proposition}
\begin{proof}
Let $f \in C(X, \A)$. Let $\varepsilon>0$. As $X$ is compact, $f$ is uniformly continuous, and thus there exists $\delta>0$ such that: 
\begin{equation}\label{deltaepsilon}
\mathsf{d}_X(x,y)<\delta\qquad\implies\qquad\|f(x)-f(y)\|_\A <\varepsilon/2
\end{equation}
for all $x,y \in X$.  Define $U(y,\delta/2)=\{x \in X: \mathsf{d}_X (x,y)<\delta/2\}$ for all $y \in X$.  Again, as $X$ is compact, the open cover $\{ U(y,\delta/2)\subseteq X: y \in X\}$ of $X$ has a finite subcover of $X$ given by $y_1, \ldots, y_M\in X$ such that $\cup_{k=1}^M U(y_k, \delta/2)=X$.  Since $X$ is compact Hausdorff, there exists a partition of unity with respect to the cover $\{U(y_1, \delta/2), \ldots, U(y_M, \delta/2)\}$ by \cite[Proposition IX.4.3.3]{Bourbaki-top2}.  In particular, for each $k \in \{1, \ldots, M\}$, there exists a continuous function $p_k: X \rightarrow [0,1]$ such that $\{x \in X: p_k(x)>0\}\neq \emptyset$ and if we define $V_k=\{x \in X: p_k(x)>0\}$, then   $\{V_1, \ldots, V_M\}$ is an open cover of $X$ and $V_k \subseteq \overline{V_k}^{\mathsf{d}_X} \subseteq U(y_k,\delta/2)$ for each $k \in \{1, \ldots, M\}$. Futhermore, we have $\sum_{k=1}^M p_k=1_{C(X)}$, which is the constant $1$ function on $X$.

Now, for each $k \in \{1, \ldots, M\}$, fix $x_k\in V_k$. Since $\overline{\cup_{n \in \N} \A_n}^{\|\cdot\|_\A}=\A$,  for each $k \in \{1, \ldots, M\}$, there exists $N_k \in \N$  and $a_k \in \A_{N_k}$ and $\|f(x_k)-a_k\|_\A <\ee/4.$  Set $N=\max\{N_1, \ldots, N_M\}$.  Now for each $k \in \{1, \ldots, M\}$, consider the function $a_k\cdot p_k \in C(X,\A_N)$, and define $f_\ee=\sum_{k=1}^M a_k \cdot p_k \in C(X, \A_N) \subseteq  C(X, \A)$. 

Next, let $x \in X$, then $x \in V_k$ for some $k \in \{1, \ldots , M\}$.  Thus $\dd_X(x,x_k)\leq \dd_X(x,y_k) +\dd_X(y_k,x_k)<\delta/2+\delta/2=\delta$, and so $\|f(x)-f(x_k)\|_\A< \ee/2$ by \eqref{deltaepsilon}.

Now, we have \begin{align*}
\|f(x)-f_\ee(x)\|_\A & =\left\|f(x) \cdot \sum_{k=1}^M p_k(x) - \sum_{k=1}^M a_k \cdot p_k(x)\right\|_\A \\& = \left\|  \sum_{k=1}^M (f(x)-a_k) \cdot p_k(x)\right\|_\A \leq \sum_{k=1}^M \|f(x)-a_k\|_\A\cdot  p_k(x)\\
& \leq  \sum_{k=1}^M \left( \|f(x)-f(x_k)\|_\A+\|f(x_k)-a_k\|_\A \right)\cdot p_k(x)\\
& <\sum_{k=1}^M \left( \ee/2+\ee/4 \right)\cdot p_k(x) = \sum_{k=1}^M (3/4)\ee\cdot  p_k(x)=(3/4)\ee.
\end{align*}
Hence $\|f-f_\ee\|_{C(X, \A)} \leq (3/4)\ee<\ee$.
Therefore 
\begin{equation}\label{eq:classical-dense}
C(X, \A)= \overline{\cup_{n \in \N} C(X, \A_n)}^{\|\cdot\|_{C(X, \A)}}. 
\end{equation}
Finally, if $\A$ is AF, and thus we assume that $\dim(\A_n)< \infty$ for each $n \in \N$, then $C(X, \A)$ is AH by  \cite[Definition V.2.1.9]{Blackadar06} since $C(X, \A_n)$ is a homogeneous C*-algebra for each $n \in \N$ by \cite[IV.1.4]{Blackadar06}. 
\end{proof}

Next, one of the key tools in providing estimates in propinquity are   conditional expectations.  For instance, we used conditional expectations in \cite{Aguilar-Latremoliere15} to obtain our quantum metrics and convergence results. So, our first task is to extend these conditional expectations on AF-algebras to ones on $C(X, \A)$. Since there are many characterizations of conditional expectations, so we list one of them here and will cite other characterizations in proofs. 

\begin{definition}[{\cite[Definition 1.5.9 and Tomiyama Theorem  1.5.10]{Brown-Ozawa}}]\label{d:cond-exp}
Let $\A, \B$ be a C*-algebras such that $\B$ is a C*-subalgebra of $\A$. A {\em projection} from $\A$ to $\B$ is a linear map $E:\A \rightarrow \B$ such that $E(b)=b$ for all $b \in \B$. A {\em conditional expectation} $E$ from $\A$ onto $\B$ is projection from $\A$ onto $\B$ such that $E$ is contractive.
\end{definition}

\begin{theorem}\label{t:ah-cond-exp}
Let $(X, \mathsf{d}_X)$ be a compact metric space and let $\A=\overline{\cup_{n \in \N} \A_n}^{\|\cdot\|_\A}$ be a unital C*-algebra equipped with faithful tracial state $\tau$ such that $\A_n$ is a finite-dimensional   C*-subalgebra of $\A$ for all $n \in \N$ and $\A_0=\C1_\A \subseteq \A_1 \subseteq \A_2 \subseteq \cdots $. In particular, $\A$ is AF.

For each $n \in \N$, let 
\[
E_{\tau,n} : \A \rightarrow \A_n
\]
be the unique $\tau$-preserving conditional expectation onto $\A_n$ given by \cite[Theorem 3.5]{Aguilar-Latremoliere15}.

If, for each $n \in \N$, we define
\[
E^X_{\tau,n} : C(X, \A) \rightarrow C(X, \A_n)
\]
by $E^X_{\tau,n}(f)(x)=E_{\tau,n}(f(x))$ for all $f \in C(X, \A)$ and $x \in X$, then:
\begin{enumerate}
\item $E^X_{\tau,n}$ is a conditional expectation onto $C(X, \A_n)$, 
\item for each $N,M \in \N$, it holds that $E^X_{\tau,N}\circ E^X_{\tau,M}=E^X_{\tau,\min\{N,M\} }=E^X_{\tau,M}\circ E^X_{\tau,N},$ and
\item for every $x \in X$, it holds that $\tau_x : f\in  C(X, \A)\mapsto \tau(f(x)) \in \C$ is a state on $C(X, \A)$ such that $\tau_x\circ E^X_{\tau,n} =\tau_x$ for all $n \in \N$. 
\end{enumerate} 
\end{theorem}
\begin{proof}
Fix $N \in \N$. Let's first show that $E^X_{\tau, N} $  is well-defined. Let $f \in C(X, \A)$.  Assume that $(x_\mu)_{\mu \in \Delta}$ is a net in $X$ that  converges to $x \in X$.  Then, if $\mu \in \Delta$, we have by contractivity of   conditional expectations  
\begin{align*}
\left\|E^X_{\tau, N} (f)(x_\mu)-E^X_{\tau, N} (f)(x)\right\|_\A& = \|E_{\tau, N} (f(x_\mu))-E_{\tau, N} (f(x))\|_\A\\
& =\|E_{\tau, N}(f(x_\mu)-f(x))\|_\A\\
& \leq \|f(x_\mu)-f(x)\|_\A.
\end{align*}
Hence, by continuity of $f$ we have that $ E^X_{\tau, N}(f) \in C(X, \A)$ and since $E_{\tau,N} (f(x)) \in \A_N$ for all $x\in X$, we have $ E^X_{\tau, N}(f) \in C(X, \A_N).$  

Note that $ E^X_{\tau, N}$ is linear by construction.  Now, we check surjectivity and that $ E^X_{\tau, N}$ projects onto $C(X, \A_N).$  Let $f \in C(X, \A_N)$.  Then, $f \in C(X,\A)$, and since $E_{\tau,N}$ projects onto $\A_N$, we have that $ E^X_{\tau, N}(f)(x)=E_{\tau,N} (f(x)) =f(x)$ for all $x \in X$ and so $ E^X_{\tau, N}(f)=f$.  

Next, let's check contractivity of $ E^X_{\tau, N}$.  Let $f \in C(X, \A)$, then for all $x \in X$, we have
\begin{align*}
\|E^X_{\tau, N}(f)(x)\|_\A= \|E_{\tau,N} (f(x))\|_\A\leq \|f(x)\|_\A \leq \|f\|_{C(X, \A)}. 
\end{align*}
Thus $\|E^X_{\tau, N}(f)\|_{C(X,\A)} \leq \|f\|_{C(X, \A)}$, and we note that 
\[\|E^X_{\tau, N}(1_{C(X,\A)})(x)\|_\A=\|E_{\tau,N} (1_\A)\|_\A=\|1_\A\|_\A=\|1_{C(X,\A)}(x)\|_{\A},\] and thus $\|E^X_{\tau, N}(1_{C(X,\A)})\|_{C(X, \A)}=\|1_{C(X,\A)}\|_{C(X,\A)}.$ Therefore by Definition \ref{d:cond-exp}, we have that $E^X_{\tau, N}$ is a conditional expectation onto $C(X,\A_N)$. 

Next, we verify (2). We note that if $N,M \in \N$, then if $f \in C(X, \A)$, then  
\begin{align*}E^X_{\tau, N}(E^X_{\tau, M}(f))(x)&=E_{\tau,N}(E^X_{\tau, M}(f)(x))=E_{\tau,N}(E_{\tau, M}(f(x)))= E_{\tau,\min\{N,M\}}(f(x))\\&= E^X_{\tau,\min\{N,M\}}(f)(x)\end{align*} for all $x \in X$ by \cite[Step 2 of proof of Theorem 3.5]{Aguilar-Latremoliere15}.

Hence $E^X_{\tau, N}\circ E^X_{\tau, M}(f)=E^X_{\tau,\min\{N,M\}}(f)$ for all $f \in C(X, \A)$ and thus $E^X_{\tau, N}\circ E^X_{\tau, M}=E^X_{\tau,\min\{N,M\}}=E^X_{\tau, M}\circ E^X_{\tau, N}$ by switching the roles of $N$ and $M$.

Finally, we establish (3). Let $n \in \N$.  Let $x  \in X$. It is routine to check that $\tau_x$ is a state on $C(X, \A)$ and follows some of the same arguments above since $C(X, \A)$ is unital. Let $f \in C(X, \A)$. We have since $E_{\tau,n}$ is $\tau$-preserving
\begin{align*}
\tau_{x}\left(E^X_{\tau,n}(f)\right)=\tau\left(E^X_{\tau,n}(f)(x)\right)=\tau\left(E_{\tau,n}(f(x))\right)=\tau(f(x))=\tau_x(f),
\end{align*}
and thus $\tau_{x}\circ E^X_{\tau,n}=\tau_x$.
\end{proof}

In order to move forward we need new quantum metrics on $C(X, \A)$ for $\A$ finite-dimensional like the ones from \cite{Aguilar-Bice17}. We present this definition using an arbitrary unital C*-algebra $\A$, and then later show that the following seminorms induce a quantum metric if and only if $\A$ is finite-dimensional as done in \cite{Aguilar-Bice17}.

\begin{definition}\label{d:h-cqms}
 Let $(X, \mathsf{d}_X)$ be a compact metric space and let $\A$ be a unital  C*-algebra.
 
 For all $f \in \sa{C(X, \A)}$ define
 \[
 l^\A_{\mathsf{d}_X}(f)= \sup_{x,y \in X} \frac{\|f(x)-f(y)\|_\A}{\mathsf{d}_X(x,y)}
 \]
 where we set $\frac{0}{0}=0$ when $x=y$. 
 
 Next, if $E: C(X, \A)\rightarrow C(X, \C1_\A)$ is a conditional expectation onto  $C(X, \C1_\A)$ and $r \in \R, r>0$, then  for all $f \in \sa{C(X, \A)}$ define
 \[
 \Lip^{\A,r}_{\mathsf{d}_X,E}(f)=  l^\A_{\mathsf{d}_X}(f)+ \frac{\|f-E(f)\|_{C(X, \A)}}{r} .
 \]
\end{definition}
Now, we show that the above seminorm forms a compact quantum metric space if and only if $\A$ is finite-dimensional. %, while also showing that this seminorm recaptures the classical structure on $C(X)$ with the Lipschitz constant.  
\begin{theorem}\label{t:h-cqms}
Let $(X, \mathsf{d}_X)$ be a compact metric space and let $\A$ be a unital C*-algebra. Let $E: C(X, \A)\rightarrow C(X, \C1_\A)$ be a conditional expectation onto $C(X, \C1_\A).$ Let $r \in \R, r>0$. The following are equivalent:
\begin{enumerate}
\item $\left(C(X, \A), \Lip_{\mathsf{d}_X,E}^{\A,r}\right)$ is a $2$-quasi-Leibniz compact quantum metric space;
\item   $\A$ is finite-dimensional.
\end{enumerate}
\end{theorem}
\begin{proof}
We begin with the reverse direction. First, we note that $l^\A_{\mathsf{d}_X}$ is a $2$-quasi-Leibniz seminorm since it is Leibniz by \cite[Proposition 2.4]{Aguilar-Bice17} and $E$ is linear.  Also the expression $\frac{\|(\cdot)-E(\cdot)\|_{C(X, \A)}}{r}$ is $2$-quasi-Leibniz by \cite[Lemma 3.2]{Aguilar-Latremoliere15}. And, the supremum of $2$-quasi-Leibniz seminorms is still $2$-quasi-Leibniz. Next, we check lower semicontinuity.  The expression $l^{\A}_{\mathsf{d}_X} $ is lower semicontinuous   by \cite[Lemma 2.6]{Aguilar-Bice17}.  Also, the expression $\frac{\|(\cdot)-E^X_{\tau,n}(\cdot)\|_{C(X, \A)}}{r}$ is continuous since the norm is continuous and $E^X_{\tau,n}$ is continuous.  Hence  $\Lip^{\A, \beta}_{\mathsf{d}_X, \mathcal{I}, \tau}$ is the supremum of lower semicontinuous maps and is thus lower semicontinuous.

Next, we check that $\{a \in \sa{\A}: \Lip_{\mathsf{d}_X,E}^{\A,r}(a)=0 \}=\R1_{C(X, \A)}$.   Let $x,y \in X$, then
\begin{align*}
\left\| 1_{C(X, \A)}(x)- 1_{C(X, \A)}(y)\right\|_{\A} =\left\|1_\A-1_\A\right\|_{\A}=0,
\end{align*}
and for all $k \in \N$, we have
\[
\left\|1_{C(X, \A)}-E (1_{C(X, \A)})\right\|_{C(X, \A)} = \left\| 1_{C(X, \A)}-1_{C(X, \A)}\right\|_{C(X, \A)}=0
\]
by Theorem \ref{t:ah-cond-exp}. Thus, $\Lip_{\mathsf{d}_X,E}^{\A,r}(1_{C(X, \A)})=0$. Next, assume that $f \in \sa{C(X, \A)}$ such that $\Lip_{\mathsf{d}_X,E}^{\A,r}(f)=0$. Hence, we have that $\|f-E (f)\|_{C(X, \A)}=0$.  Therefore, we have  that 
\begin{equation}\label{eq:kernel}f=E(f) \in  C(X, \C1_\A).
\end{equation}  Thus, for each $x \in X$, we have that $f(x)=r_x1_\A$ for some $r_x \in \R$. Hence, fix $x_0 \in X$ and let $y \in X$,
\begin{align*}
0&=\left\|  f(x_0)- f(y)\right\|_{\A}  = \left\|r_{x_0}1_\A-r_y1_\A\right\|_{\A} \\
&= |r_{x_0}-r_y|\cdot \|1_\A\|_{\A}=|r_{x_0}-r_y|.
\end{align*}
Therefore, $r_{x_0}=r_y$. 
Thus, $f(x)=r_{x_0}1_\A$ for all $x \in X$.  Hence, we have that $f=r_{x_0} 1_{C(X, \A)}$.  Therefore, 
\[
\left\{f \in \sa{C(X, \A)}:\Lip_{\mathsf{d}_X,E}^{\A,r}(f)=0\right\}=\R1_{C(X, \A)}.
\]
Note that the above arguments did not require finite-dimensionality for $\A$ as seen in the hypotheses in the referenced results, and we make this note here for a later proof.

   Since $\A$ is finite-dimensional, we have that the seminorm, defined for all $a \in C(X, \A)$ by
\[
L(a)=\max \left\{ l^\A_{\mathsf{d}_X}(a), \  \|a+C(X, \C1_\A)\|_{C(X, \A)/C(X, \C1_\A)}  \right\}, 
\]
where $\|(\cdot)+C(X, \C1_\A)\|_{C(X, \A)/C(X, \C1_\A)} $ is the quotient norm, is a Lip-norm on $C(X, \A)$  by \cite[Theorem 2.10]{Aguilar-Bice17}. By construction, we have  $\dom{L}=\dom{\Lip_{\mathsf{d}_X,E}^{\A,r}}$, and thus, we also have that $\dom{\Lip_{\mathsf{d}_X,E}^{\A,r}}$ is dense.   Since $E(a) \in C(X, \C1_\A)$ for all $a \in C(X, \A)$, we have that 
\[
L \leq \Lip_{\mathsf{d}_X,E}^{\A,1}.
\]
Hence, by \cite[Comparison Lemma 1.10]{Rieffel98a}, we have that $\left( C(X, \A),  \Lip_{\mathsf{d}_X,E}^{\A,1}\right)$ is a compact quantum metric space. 

Now, assume that $0<r<1$.  Then, by construction, we have that 
\[
\Lip_{\mathsf{d}_X,E}^{\A,1} \leq \Lip_{\mathsf{d}_X,E}^{\A,r},
\]
and so $\left( C(X, \A),  \Lip_{\mathsf{d}_X,E}^{\A,r}\right)$ is a compact quantum metric space by the same argument. Next, assume that $r \geq 1$.  Note that it is easily verified that $\frac{1}{r} \Lip_{\mathsf{d}_X,E}^{\A,1}$ is a Lip-norm. By construction, we have that 
\[
\frac{1}{r}\Lip_{\mathsf{d}_X,E}^{\A,1} \leq \Lip_{\mathsf{d}_X,E}^{\A,r},
\]
and so $\left( C(X, \A),  \Lip_{\mathsf{d}_X,E}^{\A,r}\right)$ is a compact quantum metric space by the same argument.

The forward direction follows the same proof as \cite[Theorem 2.10]{Aguilar-Bice17} up to scaling by $\frac{1}{r}$  and using the fact that conditional expectations are contractive   just as states are.
\end{proof}

\begin{remark}
We note that the above places a compact quantum metric on any C*-algebra of the form $C(X, \A)$ for $(X, \mathsf{d}_X)$ compact metric and $\A$ finite-dimensional C*-algebra. Indeed, by finite-dimensionality and existence of faithful tracial state, there exists a conditional expectation from $\A$ onto $\C1_\A$, which can be extended to a conditional expectation from $C(X, \A)$ onto $C(X, \C1_\A)$ as seen in  Theorem \ref{t:ah-cond-exp}. 
\end{remark}

We are almost ready to present the main  quantum metrics of this article. We note that the above Theorem shows that we MUST do more in the case when $\A$ is infinite dimensional if we still want to include the Lipschitz constant in this manner, and this manner is desirable since it provides a Leibniz rule for elementary tensors as seen in Expression \eqref{eq:leibniz-tensor}. The method to remedy this is to use the Lip-norms from Definition \ref{d:h-cqms} along with techniques from \cite{Aguilar-Latremoliere15, Latremoliere05, Kerr09}. Furthermore, we will see that we recover all structure from previous and classical structure while showing that our construction of Lip-norm satisfies a  Leibniz-type rule on elementary tensors, when $C(X, \A)$ is viewed as $C(X) \otimes \A$, which establishes that our construction is compatible with the tensor product structure in Expression \eqref{eq:leibniz-tensor}. To motivate why we say that this is compatible with the tensor product structure, we present a classical case of tensor products.  First, we introduce some notation and prove a classical lemma.
\begin{notation}\label{n:prod-metric}
Let $(X, \mathsf{d}_X)$ and $(Y, \mathsf{d}_Y)$ be compact metric spaces. We denote the $1$-metric on the Cartesian product $X \times Y$ by $\mathsf{d}^1_{\mathsf{d}_X \times \mathsf{d}_Y}$, where for all $(x_1, y_1), (x_2, y_2) \in X \times Y$, we have 
\[
\mathsf{d}^1_{\mathsf{d}_X \times \mathsf{d}_Y}((x_1, y_1), (x_2, y_2))=\mathsf{d}_X(x_1, x_2)+\mathsf{d}_Y(y_1, y_2).
\]
\end{notation}
\begin{lemma}\label{l:pure-state-norm-comm}
If $\A$ is a unital commutative C*-algebra, then
\[
\|a\|_\A=\sup_{\nu \in \mathscr{PS}(\A)} |\nu(a)|
\]
for all $a \in \A$, where $\mathscr{PS}$ denotes pure states of $\A$.
\end{lemma}
\begin{proof}
Let $a \in \A$.   Note that 
\[
\|a\|^2_\A=\|a^*a\|_\A=\sup_{\nu \in \mathscr{PS}(\A)}  \nu(a^*a)
\]
by \cite[Lemma I.9.10]{Davidson} since $a^*a$ is positive. Thus $\|a\|_\A=\sup_{\nu \in \mathscr{PS}(\A)}  \sqrt{\nu(a^*a)}.$  Next, since pure states on unital commutative C*-algebras  are multiplicative by \cite[Proposition 4.4.1]{Kadison97}, we have that
\begin{align*}
\|a\|_\A& = \sup_{\nu \in \mathscr{PS}(\A)}  \sqrt{\nu(a^*)\nu(a)}= \sup_{\nu \in \mathscr{PS}(\A)}  \sqrt{\overline{\nu(a)}\nu(a)}\\
& = \sup_{\nu \in \mathscr{PS}(\A)}  \sqrt{|\nu(a)|^2}=\sup_{\nu \in \mathscr{PS}(\A)}   |\nu(a)|,
\end{align*}
which completes the proof.
\end{proof}

\begin{theorem}\label{t:comm-tensor}
Let $X,Y$ be compact Hausdorff spaces. It holds that 
\[
C(X \times Y)\cong C(X) \otimes C(Y),
\]
where $C(X) \otimes C(Y)$ is the C*-algebra formed over the tensor product of C*-algebras given by \cite[Chapter 3]{Brown-Ozawa}, which is unique by \cite[Proposition 2.4.2 and Proposition 3.6.12]{Brown-Ozawa}.

In particular, there exists a unique *-isomorphism
\[
c_{T} :   C(X) \otimes C(Y) \rightarrow C(X \times Y)
\]
such that for all $f \in C(X), g \in C(Y)$, it holds that $c_{T}(f \otimes g)(x,y)=f(x)g(y)$ for all $x\in X,y \in Y$.  

If, furthermore, $(X, \mathsf{d}_X)$, $(Y, \mathsf{d}_Y)$ are compact metric spaces, then using notation from Definition \ref{d:h-cqms} and Notation \ref{n:prod-metric}, we have that 
\begin{equation}\label{eq:leibniz-tensor-comm}
\begin{split}
l^\C_{\mathsf{d}^1_{\mathsf{d}_X \times \mathsf{d}_Y}}(c_T(f \otimes g))&  \leq l^\C_{\mathsf{d}_X}(f)\|g\|_{C(Y)}+l^\C_{\mathsf{d}_Y}(g)\|f\|_{C(X)}\\
& = l^\C_{\mathsf{d}_X}(f)\otimes \|g\|_{C(Y)}+l^\C_{\mathsf{d}_Y}(g)\otimes \|f\|_{C(X)}
\end{split}
\end{equation}
for all $f\in C(X), g\in C(Y).$
\end{theorem}
\begin{proof}
The result $C(X \times Y)\cong C(X) \otimes C(Y)$ is well-known but may be difficult to find in the literature, and thus, we provide a proof.

Let $C(X)\odot C(Y)$ denote the algebra over $\C$ formed over the algebraic tensor product of $C(X)$ and $C(Y)$ \cite[Section 3.1]{Brown-Ozawa}, which is dense in  $ C(X) \otimes C(Y)$ by definition of the C*-algebraic tensor product. For all $f \in C(X), g \in C(Y)$, define
 \[
 c_T(f \otimes g): (x,y) \in X \times Y \mapsto f(x)g(y) \in \C
\]
		and we note that $c_T(f \otimes g) \in C(X \times Y)$.  Now, extend $c_T$ to a *-homomorphism on $C(X)\odot C(Y)$.  Now, we show $c_T$ is an isometry on $C(X)\odot C(Y)$.
		
		Let $a\in C(X) \odot C(Y)$. 
		Note that 
		\begin{equation}\label{eq:pure-norm-comm}\|a\|_{C(X) \otimes C(Y)}=\sup_{\nu \in \mathscr{PS}( C(X) \otimes C(Y))} |\nu(a)|,\end{equation}
	by 	Lemma \ref{l:pure-state-norm-comm} since $C(X)\otimes C(Y)$ is unital and commutative.
		
		 Now, since $a\in C(X) \odot C(Y)$, there exist $n \in \N$, $r_0, \ldots, r_n \in \C, f_0, \ldots, f_n \in C(X), g_0, \ldots, g_n \in C(Y)$ such that $a=\sum_{k=0}^n r_k (f_k\otimes g_k)$.  Let $\mu$ be a pure state on $C(X) \otimes C(Y)$. Then, there exist a pure state $\mu_X$ on $C(X)$ and a pure state $\mu_Y$ on $C(Y)$ such that $\mu(f \otimes g)=\mu_X(f)\mu_Y(g)$ for all $f \in C(X), g \in C(Y)$ by   \cite[Corollary 3.4.3]{Brown-Ozawa}.  Now, by the beginning of the proof of Theorem \ref{t:comm-mk-diam}, we have that $\mu_X=\delta_x$ and $\mu_Y=\delta_y$ for some $x \in X,y \in Y$, so that $\mu(f \otimes g)=\delta_x(f)\delta_y(g)=f(x)g(y)=c_T(f\otimes g)(x,y)$ for all $f \in C(X), g \in C(Y)$. Therefore, we have that
	\begin{align*}
	|\mu(a)|& = \left|\sum_{k=0}^n r_k \cdot \mu (f_k\otimes g_k)\right|=  \left|\sum_{k=0}^n r_k \cdot c_T(f_k\otimes g_k)(x,y)\right|\\
	& = \left|c_T\left(\sum_{k=0}^n r_k (f_k\otimes g_k)\right)(x,y)\right|=\left|c_T(a)(x,y)\right| \leq \|c_T(a)\|_{C(X \times Y)},
	\end{align*}
	which implies that $\|a\|_{C(X)\otimes C(Y)}\leq \|c_T(a)\|_{C(X \times Y)}$ by Expression \eqref{eq:pure-norm-comm} and since $\mu$ was an arbitrary pure state.
	
	 For the other inequality, begin with $(x,y) \in X \times Y$. Now, since $\delta_x$ and $\delta_y$ are pure states, there exists a pure state $\mu$ on $C(X) \otimes C(Y)$ such that $\mu( f \otimes g)=\delta_x(f)\delta_y(g)$ for all $f \in C(X), g \in C(Y)$ by \cite[Theorem 6.4.13]{Murphy90} since the C*-norm on $C(X) \otimes C(Y)$ is the spatial/min C*-norm by nuclearity by commutativity \cite[Takesaki Theorem 6.4.15]{Murphy90}. And, the same argument above shows that
	 \[
	 |c_T(a)(x,y)|=|\mu(a)|\leq \|a\|_{C(X)\otimes C(Y)}
	 \]
	since states are contractive. And thus $\|c_T(a)\|_{C(X \times Y)} \leq \|a\|_{C(X)\otimes C(Y)}$ since $(x,y) \in X \times Y$ was arbitrary. 
	 
	 Therefore $c_T$ is an isometry and thus by density and completeness, $c_T$ extends to an isometric *-homomorphism on $C(X) \otimes C(Y)$, which we  still denote by $c_T$. Now, $c_T(C(X) \otimes C(Y))$ is a unital *-subalgebra of $C(X\times Y)$ by construction. Fix $(x,y), (x',y') \in X \times Y$. We will only do the case when $x\neq x'$ and $y\neq y'$, and the other cases follow similarly. By \cite[Urysohn's Lemma 15.6]{Willard}, there exists $f \in C(X)$ such that $f(x)=1$ and $f(x')=0$, and there exists $g \in C(Y)$ such that $g(y)=1$ and $g(y')=0$.  Hence  $c_T(f\otimes g)(x,y)=f(x)g(y)=1\neq 0=f(x')g(y')=c_T(f\otimes g)(x',y').$  Thus $c_T(C(X) \otimes C(Y))$ also separates points and thus by \cite[Stone-Weierstrass Theorem V.8.1]{Conway90}, we have that $c_T(C(X) \otimes C(Y))$ is dense in $C(X\times Y)$.  Therefore, as $c_T$ is an isometry, we have that $c_T(C(X) \otimes C(Y))=C(X\times Y)$ by completeness. The uniqueness follows by density and continuity.

	Finally, we assume that $(X, \mathsf{d}_X)$ and $(Y, \mathsf{d}_Y)$ are compact metric spaces. Let $f\in C(X), g \in C(Y)$.  Let $(x,y),(x',y') \in X \times Y$.  We then have
	\begin{align*}
	\frac{|c_T(f\otimes g)(x,y)-c_T(f\otimes g)(x',y')|}{\mathsf{d}^1_{\mathsf{d}_X \times \mathsf{d}_Y}((x,y),(x',y'))}& = \frac{| f(x)g(y)-f(x')g(y')|}{\mathsf{d}_X(x,x')+\mathsf{d}_Y(y,y')}\\
	& \leq \frac{| f(x)-f(x')|\cdot|g(y)|}{\mathsf{d}_X(x,x')+\mathsf{d}_Y(y,y')}\\
	& \quad +\frac{|f(x')|\cdot|g(y)-g(y')|}{\mathsf{d}_X(x,x')+\mathsf{d}_Y(y,y')}\\
	& \leq \frac{| f(x)-f(x')|\cdot|g(y)|}{\mathsf{d}_X(x,x') }\\
	& \quad +\frac{|f(x')|\cdot|g(y)-g(y')|}{ \mathsf{d}_Y(y,y')}\\
	& \leq l^\C_{\mathsf{d}_X}(f)\|g\|_{C(Y)}+ l^\C_{\mathsf{d}_Y}(g)\|f\|_{C(X)},
	\end{align*}
	which completes the proof.
	\end{proof}

We now present the quantum metrics we will study for the rest of this article.  These quantum metrics translate a standard categorical limit into a  metric limit in propinquity of the inductive sequence while also providing a noncommutative analogue to the tensor Leibniz rule of Expression \eqref{eq:leibniz-tensor-comm} in Expression \eqref{eq:leibniz-tensor}.
\begin{theorem}\label{t:ah-c*-cqms}
Let $(X, \mathsf{d}_X)$ be a compact metric space and let $\A=\overline{\cup_{n \in \N} \A_n}^{\|\cdot\|_\A}$ be a unital C*-algebra equipped with faithful tracial state $\tau$ such that $\A_n$ is a finite-dimensional  C*-subalgebra of $\A$ for all $n \in \N$ and $\A_0=\C1_\A \subseteq \A_1 \subseteq \A_2 \subseteq \cdots $.  In particular, $\A$ is AF. Denote $\mathcal{I}=(\A_n)_{n \in \N}$. Let $(\beta(n))_{n \in \N}$ be a sequence of positive real numbers that converges to $0$. 

Using notation from Definition \ref{d:h-cqms} and Theorem \ref{t:ah-cond-exp}, if we define
\[
\Lip^{\A, \beta}_{\mathsf{d}_X, \mathcal{I}, \tau} (f) = l^{\A}_{\mathsf{d}_X}\left( f \right)+ \sup_{n \in \N} \left\{ \frac{\left\|f-E^X_{\tau,n}(f)\right\|_{C(X, \A)}}{\beta(n)} \right\} 
\]
 for all $f \in \sa{C(X, \A)}$, then:
 \begin{enumerate}
 \item $\left(C(X, \A_n), \Lip^{\A, \beta}_{\mathsf{d}_X, \mathcal{I}, \tau} \right)$ is a $2$-quasi-Leibniz compact quantum metric space for each $n \in \N$, where 
 \begin{equation}\label{eq:af-h}
 \Lip^{\A, \beta}_{\mathsf{d}_X, \mathcal{I}, \tau}  (f) =   l^{\A_n}_{\mathsf{d}_X}\left( f \right)+  \max_{k \in \{0, \ldots, n\}}\left\{\frac{\left\|f-E^X_{\tau,k}(f)\right\|_{C(X, \A)}}{\beta(k)}\right\} 
 \end{equation}
 for all $f \in \sa{C(X, \A_n)}$,
 \item $\left(C(X, \A), \Lip^{\A, \beta}_{\mathsf{d}_X, \mathcal{I}, \tau} \right)$ is a $2$-quasi-Leibniz compact quantum metric space, 
 \item $\qpropinquity{} \left(\left(C(X, \A), \Lip^{\A, \beta}_{\mathsf{d}_X, \mathcal{I}, \tau} \right),\left(C(X, \A_n), \Lip^{\A, \beta}_{\mathsf{d}_X, \mathcal{I}, \tau} \right) \right) \leq  \beta(n)$, for each $n \in \N$ and thus,
 \[\lim_{n \to \infty} \qpropinquity{} \left(\left(C(X, \A), \Lip^{\A, \beta}_{\mathsf{d}_X, \mathcal{I}, \tau} \right),\left(C(X, \A_n), \Lip^{\A, \beta}_{\mathsf{d}_X, \mathcal{I}, \tau} \right) \right)=0.\]
 \item  if $\A=\C$, then $\left(C(X, \A),\Lip^{\A, \beta}_{\mathsf{d}_X, \mathcal{I}, \tau} \right)$ is fully quantum isometric  (in the sense of Theorem \ref{t:distq}) to $\left(C(X), l^\C_{\mathsf{d}_X}\right)$,
 \item  if $X=\{x\}$, then $\left(C(X, \A), \Lip^{\A, \beta}_{\mathsf{d}_X, \mathcal{I}, \tau} \right)$ is fully quantum isometric  (in the sense of Theorem \ref{t:distq}) to $\left(\A, \Lip^\beta_{\mathcal{I}, \tau}\right)$ of Theorem \ref{AF-lip-norms-thm-union}, and 
 \item  if $\dim(\A)< \infty$ and we set   $\A_n=\A$ for all $n \in \N \setminus \{0\}$ and   $\beta(0)=r$ for some $r\in \R, r>0$, then $\left(C(X, \A),\Lip^{\A, \beta}_{\mathsf{d}_X, \mathcal{I}, \tau} \right)$ is fully quantum isometric  (in the sense of Theorem \ref{t:distq}) to  $\left(C(X,\A), \Lip^{\A, r}_{\mathsf{d}_X, E^X_{\tau, 0}}\right)$ of Definition \ref{d:h-cqms}.
 \end{enumerate}
 Furthermore, if we let $C(X) \otimes \A$ be the C*-algebraic tensor product over $C(X)$ and $\A$, which is unique by  \cite[Proposition 2.4.2 and Proposition 3.6.12]{Brown-Ozawa}, and we let $\pi_{X, \A} : C(X) \otimes \A \rightarrow C(X, \A)$ be the canonical *-isomorphism of \cite[Theorem 6.4.17]{Murphy90} such that for all $f \in C(X), a \in \A$, it holds that $\pi_{X, \A}(f\otimes a)(x)=f(x)\cdot a$ for all $x \in X$, and we denote 
  \begin{equation}\label{eq:tensor-lip} 
  \Lip^{\A, \beta, \otimes}_{\mathsf{d}_X, \mathcal{I}, \tau}=\Lip^{\A, \beta}_{\mathsf{d}_X, \mathcal{I}, \tau}\circ \pi_{X, \A},
  \end{equation} then (1)-(5) all hold with $C(X, \A)$ and $C(X, \A_n)$ replaced by $C(X)\otimes \A$ and $C(X) \otimes \A_n$, respectively, for all $n \in \N$, and  $\Lip^{\A, \beta}_{\mathsf{d}_X, \mathcal{I}, \tau}$ replaced by $\Lip^{\A, \beta, \otimes}_{\mathsf{d}_X, \mathcal{I}, \tau}$ (except for Expression \eqref{eq:af-h}) and (6) holds with only $\left(C(X, \A), \Lip^{\A, \beta}_{\mathsf{d}_X, \mathcal{I}, \tau}\right)$ replaced with  $\left(C(X)\otimes \A, \Lip^{\A, \beta, \otimes}_{\mathsf{d}_X, \mathcal{I}, \tau}\right)$, and note that for all $f \in C(X), a \in \A$, it holds that 
 \begin{equation}\label{eq:leibniz-tensor}
 \begin{split}
 \Lip^{\A, \beta, \otimes}_{\mathsf{d}_X, \mathcal{I}, \tau}(f \otimes a)&= l^\C_{\mathsf{d}_X}(f) \|a\|_\A+   \Lip^\beta_{\mathcal{I}, \tau}(a)\|f\|_{C(X)}\\
 & = l^\C_{\mathsf{d}_X}(f)\otimes \|a\|_\A+   \Lip^\beta_{\mathcal{I}, \tau}(a)\otimes \|f\|_{C(X)},
 \end{split}
 \end{equation}
 where $\Lip^\beta_{\mathcal{I}, \tau}$ is from  Theorem \ref{AF-lip-norms-thm-union}, and thus
 \[
  \Lip^{\A, \beta, \otimes}_{\mathsf{d}_X, \mathcal{I}, \tau}(1_{C(X)} \otimes a)=   \Lip^\beta_{\mathcal{I}, \tau}(a) \text{ and }  \Lip^{\A, \beta, \otimes}_{\mathsf{d}_X, \mathcal{I}, \tau}(f \otimes 1_\A)=    l^\C_{\mathsf{d}_X}(f).
 \]
\end{theorem}
\begin{proof}
(1)  We will prove much of (2) in proving (1). First, we note that  $\Lip^{\A, \beta}_{\mathsf{d}_X, \mathcal{I}, \tau}$ is a $2$-quasi-Leibniz seminorm that is lower semicontinuous  by the same argument as Theorem \ref{t:h-cqms} since finite-dimensionality was not required for this part of the argument and the supremum of lower semicontinuous maps is again lower semicontinuous.

 Next, let's show that $\dom{\Lip^{\A, \beta}_{\mathsf{d}_X, \mathcal{I}, \tau} }$ is dense in $\sa{C(X, \A)}$.  Let $n \in \N$.  Let $f \in \dom{l^{\A_n}_{\mathsf{d}_X}}$.  Then, we have  $f \in \sa{C(X, \A_n)}$ and $l^{\A_n}_{\mathsf{d}_X}(f) < \infty$.  However, since the C*-norm on $\A_n$ is given by the C*-norm on $\A$ by assumption, we have that $l^{\A}_{\mathsf{d}_X}(f)=l^{\A_n}_{\mathsf{d}_X}(f)< \infty$. Also, for each $k \in \N, k \geq n$, we have that 
\[
\left\| f-E^X_{\tau,k}(f)\right\|_{C(X, \A)} = \left\| f-f\right\|_{C(X, \A)}=0
\]
by Theorem \ref{t:ah-cond-exp}. Hence,  \[\sup_{n \in \N}\left\{\frac{\left\|f-E^X_{\tau,n}(f) \right\|_{C(X, \A)}}{\beta(n)}\right\}=\max_{k \in \{0, \ldots, n\}}\left\{ \frac{\left\|f-E^X_{\tau,n}(f) \right\|_{C(X, \A)}}{\beta(n)}\right\}.\] Thus, $\Lip^{\A, \beta}_{\mathsf{d}_X, \mathcal{I}, \tau}(f)< \infty$ for all $f \in \dom{l^{\A_n}_{\mathsf{d}_X}}.$  Since $n \in \N$ was arbitrary, we have that
\[
\cup_{n \in \N} \dom{l^{\A_n}_{\mathsf{d}_X}} \subseteq \dom{\Lip^{\A, \beta}_{\mathsf{d}_X, \mathcal{I}, \tau}}.
\]
Now, $\dom{l^{\A_n}_{\mathsf{d}_X}}$ is dense in $\sa{C(X, \A_n)}$ for all $n \in \N$ by   \cite[Lemma 2.6]{Aguilar-Bice17}, and thus $\cup_{n \in \N} \dom{l^{\A_n}_{\mathsf{d}_X}}$ is dense in $\sa{C(X, \A)}$.  Therefore, $\dom{\Lip^{\A, \beta}_{\mathsf{d}_X, \mathcal{I}, \tau}}$ is dense in $\sa{C(X, \A)}$ by Proposition \ref{p:cx-dense}. Furthermore, we have also established Expression  \eqref{eq:af-h}.

Next, we check the kernel of $\Lip^{\A, \beta}_{\mathsf{d}_X, \mathcal{I}, \tau}$.  Let $x,y \in X$, then
\begin{align*}
\left\| 1_{C(X, \A)}(x)- 1_{C(X, \A)}(y)\right\|_{\A} =\left\|1_\A-1_\A\right\|_{\A}=0,
\end{align*}
and for all $k \in \N$, we have
\[
\left\|1_{C(X, \A)}-E^X_{\tau,k}(1_{C(X, \A)})\right\|_{C(X, \A)} = \left\| 1_{C(X, \A)}-1_{C(X, \A)}\right\|_{C(X, \A)}=0
\]
by Theorem \ref{t:ah-cond-exp}. Thus, $\Lip^{\A, \beta}_{\mathsf{d}_X, \mathcal{I}, \tau}(1_{C(X, \A)})=0$. Next, assume that $f \in \sa{C(X, \A)}$ such that $\Lip^{\A, \beta}_{\mathsf{d}_X, \mathcal{I}, \tau}(f)=0$. Hence, we have that $\|f-E^X_{\tau, 0}(f)\|_{C(X, \A)}=0$.  Therefore, we have  that $f=E^X_{\tau, 0}(f) \in  C(X, \C1_\A)$. And, the rest of the argument that \[
\left\{f \in \sa{C(X, \A)}: \Lip^{\A, \beta}_{\mathsf{d}_X, \mathcal{I}, \tau}(f)=0\right\}=\R1_{C(X, \A)}.
\] follows from the same argument after Expression \eqref{eq:kernel}.

Now, let $n \in \N$, by Expression \eqref{eq:af-h}, we have that
\[
\Lip_{\mathsf{d}_X, E_{\tau,0}^X}^{\A_n,\beta(0)} (f) \leq  \Lip^{\A, \beta}_{\mathsf{d}_X, \mathcal{I}, \tau}(f)
\]
for all $f\in \sa{C(X, \A_n)}$.  Hence, $\Lip^{\A, \beta}_{\mathsf{d}_X, \mathcal{I}, \tau}$ is a Lip-norm on $C(X, \A_n)$ by \cite[Comparison Lemma 1.10]{Rieffel98a} since $\Lip_{\mathsf{d}_X, E_{\tau,0}^X}^{\A_n,\beta(0)}$ is a Lip-norm on $C(X, \A_n)$  by Theorem \ref{t:h-cqms}. This completes (1).

(2) By the details verified in part (1) above, we only have to check that $\Kantorovich{ \Lip^{\A, \beta}_{\mathsf{d}_X, \mathcal{I}, \tau}}$ metrizes the weak* topology of $\StateSpace(C(X, \A))$.  To accomplish this, we will use Theorem \ref{Rieffel-thm}.  Fix $x_0 \in X$ and consider $\tau_{x_0} \in \StateSpace(C(X, \A))$ of Theorem \ref{t:ah-cond-exp}. We will now show that 
\[
L_{\tau_{x_0}}=\left\{ f \in \sa{C(X, \A)} : \Lip^{\A, \beta}_{\mathsf{d}_X, \mathcal{I}, \tau}(f)\leq 1 \text{ and } \tau_{x_0}(f)=0\right\}
\]
is compact.  It is already closed since $\tau_{x_0}$ is continuous and $ \Lip^{\A, \beta}_{\mathsf{d}_X, \mathcal{I}, \tau}$ is lower semicontinuous. So, we only have to show that $L_{\tau_{x_0}}$ is totally bounded. Let $\varepsilon>0$.  There exists $N \in \N$ such that $\beta(N)< \varepsilon/2$. Now, since $\left(C(X, \A_N), \Lip^{\A, \beta}_{\mathsf{d}_X, \mathcal{I}, \tau}\right)$ is a compact quantum metric space by part (1), 
we have that the set 
\[L^N_{\tau_{x_0}}=\left\{ f \in \sa{C(X, \A)} : \Lip^{\A, \beta}_{\mathsf{d}_X, \mathcal{I}, \tau}(f)\leq 1 \text{ and } \tau_{x_0}(f)=0\right\}\]
is totally bounded by Theorem \ref{Rieffel-thm} since $\tau_{x_0} \in \StateSpace(C(X, \A_N)).$ Hence, there exists a finite $\varepsilon/2$-net $\{g_0, \ldots, g_M\} \subseteq L^N_{\tau_{x_0}}$ of $L^N_{\tau_{x_0}}$. We will now show that $\{g_0, \ldots, g_M\}$ is a finite $\varepsilon$-net of $L_{\tau_{x_0}}$.  First, we note that $\{g_0, \ldots, g_M\}\subseteq L_{\tau_{x_0}}.$  Let $f \in  L_{\tau_{x_0}}$.  Note that $E^X_{\tau,N}(f) \in \sa{C(X, \A_N)}$ since conditional expectations are positive by \cite[Theorem 1.5.10 (Tomiyama)]{Brown-Ozawa}. Now, $f \in L_{\tau_{x_0}}$ implies that $\Lip^{\A, \beta}_{\mathsf{d}_X, \mathcal{I}, \tau}(f)\leq 1$, which implies that 
\begin{equation}\label{eq:cqms-est}
\left\|f-E^X_{\tau,N}(f)\right\|_{C(X,\A)} \leq \beta(N)< \varepsilon/2.
\end{equation}
Note that $0=\tau_{x_0}(f)=\tau_{x_0}(E^X_{\tau,N}(f))$ by Theorem \ref{t:ah-cond-exp}. Now, we will show that $\Lip^{\A, \beta}_{\mathsf{d}_X, \mathcal{I}, \tau}(E^X_{\tau,N}(f))\leq \Lip^{\A, \beta}_{\mathsf{d}_X, \mathcal{I}, \tau}(f)\leq 1.$

Let $n \in \N$. Let $x,y \in X$, then 
\begin{equation*}
\begin{split}
 \left\|E^X_{\tau,N}(f)(x)-E^X_{\tau,N}(f)(y)\right\|_{\A} & =  \left\|E_{\tau,N}(f(x))-E_{\tau,N}(f(y))\right\|_{\A} \\
& =  \left\|E_{\tau,N}(f(x)- f(y))\right\|_{\A}   \leq  \left\| f(x)- f(y)\right\|_{\A} 
\end{split}
\end{equation*}
by Theorem \ref{t:ah-cond-exp} and thus
\begin{equation}\label{eq:comm-lip} l^\A_{\mathsf{d}_X}( E^X_{\tau,N}(f))\leq l^{\A}_{\mathsf{d}_X}(f).\end{equation}

Let   $n \geq N$.  Then,  by Theorem \ref{t:ah-cond-exp}
\begin{align*}
\frac{\left\| E^X_{\tau,N}(f)- E^X_{\tau,n}\left(E^X_{\tau,N}(f)\right)\right\|_{C(X, \A)}}{\beta(n)} &= \frac{\left\| E^X_{\tau,N}(f)- E^X_{\tau,\min\{n,N\}}(f)\right\|_{C(X, \A)}}{\beta(n)}\\
&= \frac{\left\| E^X_{\tau,N}(f)- E^X_{\tau,N}(f)\right\|_{C(X, \A)}}{\beta(n)}\\
&=0\leq \frac{\left\|  f- E^X_{\tau,n}(f) \right\|_{C(X, \A)}}{\beta(n)}.
\end{align*}
Next, if $n \leq N-1$, then by Theorem \ref{t:ah-cond-exp}
\begin{align*}
\frac{\left\| E^X_{\tau,N}(f)- E^X_{\tau,n}\left(E^X_{\tau,N}(f)\right)\right\|_{C(X, \A)}}{\beta(n)} &= \frac{\left\| E^X_{\tau,N}(f)-  E^X_{\tau,N}\left(E^X_{\tau, n}(f)\right)\right\|_{C(X, \A)}}{\beta(n)}\\
&= \frac{\left\| E^X_{\tau,N}\left(f- E^X_{\tau,n}(f)\right)\right\|_{C(X, \A)}}{\beta(n)}\\
&\leq  \frac{\left\|  f- E^X_{\tau,n}(f) \right\|_{C(X, \A)}}{\beta(n)} .
\end{align*}
Hence
\begin{equation}\label{eq:af-lip}
\sup_{n \in \N} \left\{ \frac{\left\|  E^X_{\tau,N}(f)- E^X_{\tau,n}(E^X_{\tau,N}(f)) \right\|_{C(X, \A)}}{\beta(n)}\right\} \leq \sup_{n \in \N} \left\{ \frac{\left\|   f- E^X_{\tau,n}( f) \right\|_{C(X, \A)}}{\beta(n)}\right\}.
\end{equation}
Therefore, combining Expressions \eqref{eq:comm-lip} and \eqref{eq:af-lip}, we have  
\begin{equation}\label{eq:contractive-lip}\Lip^{\A, \beta}_{\mathsf{d}_X, \mathcal{I}, \tau}(E^X_{\tau,N}(f))\leq \Lip^{\A, \beta}_{\mathsf{d}_X, \mathcal{I}, \tau}(f),
\end{equation}
and thus $\Lip^{\A, \beta}_{\mathsf{d}_X, \mathcal{I}, \tau}(E^X_{\tau,N}(f))\leq 1.$   Hence $E^X_{\tau,N}(f) \in L^N_{\tau_{x_0}}$.  Therefore, there exists $j \in \{0, \ldots, M\}$ such that
\[
\|E^X_{\tau,N}(f)-g_j\|_{C(X, \A)} < \varepsilon/2.
\]
Combining this with Expression \eqref{eq:cqms-est}, we have that
\[
\|f-g_j \|_{C(X, \A)} \leq \|f-E^X_{\tau,N}(f)\|_{C(X, \A)} + \|E^X_{\tau,N}(f)-g_j\|_{C(X, \A)}< \varepsilon/2+\varepsilon/2=\varepsilon.
\]
Therefore, $L_{\tau_{x_0}}$ is totally bounded and thus compact.  Hence $\Lip^{\A, \beta}_{\mathsf{d}_X, \mathcal{I}, \tau}$ is a $(2,0)$-quasi-Leibniz Lip-norm on $C(X, \A)$ by Theorem \ref{Rieffel-thm}.

(3) Let $n\in \N$. Consider the bridge of Definition \ref{bridge-def}  from $C(X, \A_n)$ to $C(X, \A)$ given by $\gamma=(C(X, \A), \iota_n, \mathrm{id}_{C(X,\A)}, 1_{C(X, \A)})$.  In this case, we have that \[\StateSpace_1(C(X, \A)|1_{C(X, \A)})=\StateSpace(C(X, \A)).\] Thus, the height of this bridge (Definition \ref{d:height}) is $0.$ Thus, the length of the bridge (Definition \ref{d:length}) is equal to its reach (Definition \ref{d:reach}), which we estimate now. 

First, let $f \in \sa{C(X, \A_n)}$ such that $ \Lip^{\A, \beta}_{\mathsf{d}_X, \mathcal{I}, \tau}(f)\leq 1$.  Then $f \in \sa{C(X, \A)}$ such that $ \Lip^{\A, \beta}_{\mathsf{d}_X, \mathcal{I}, \tau}(f)\leq 1$, and $\|f-f\|_{C(X, \A)}= 0\leq \beta(n)$.

Next, let $f \in \sa{C(X, \A)}$ such that $ \Lip^{\A, \beta}_{\mathsf{d}_X, \mathcal{I}, \tau}(f)\leq 1$. Thus, by construction, we have that $\|f-E^X_{\tau,n}(f)\|_{C(X, \A)}\leq \beta(n)$. Now, $E^X_{\tau,n}(f) \in \sa{C(X, \A_n)}$ since conditional expectations are positive by \cite[Theorem 1.5.10 (Tomiyama)]{Brown-Ozawa}, and we have that $\Lip^{\A, \beta}_{\mathsf{d}_X, \mathcal{I}, \tau}(E^X_{\tau,n}(f))\leq \Lip^{\A, \beta}_{\mathsf{d}_X, \mathcal{I}, \tau}(f)\leq 1$ by Expression \eqref{eq:contractive-lip}. Then $f \in \sa{C(X, \A)}$ such that $ \Lip^{\A, \beta}_{\mathsf{d}_X, \mathcal{I}, \tau}(f)\leq 1$.

Thus, we have that the reach $\varrho(\gamma|\Lip^{\A, \beta}_{\mathsf{d}_X, \mathcal{I}, \tau},\Lip^{\A, \beta}_{\mathsf{d}_X, \mathcal{I}, \tau})\leq \beta(n)$  by definition of the Hausdorff distance.  Thus, the length $\lambda(\gamma|\Lip^{\A, \beta}_{\mathsf{d}_X, \mathcal{I}, \tau},\Lip^{\A, \beta}_{\mathsf{d}_X, \mathcal{I}, \tau})\leq \beta(n).$ Therefore, by Theorem \ref{t:distq}, part (3) is complete since $(\beta(n))_{n \in \N}$ is assumed to converge to $0$.

(4) In this case, we have that $f=E^X_{\tau,n}(f)$ for all $n \in \N$, and the result follows, and in fact, we have $\left(C(X, \A),\Lip^{\A, \beta}_{\mathsf{d}_X, \mathcal{I}, \tau} \right)=\left(C(X), l^\C_{\mathsf{d}_X}\right)$. 

(5) Consider the *-isomorphism given by $a \in \A \mapsto (x \mapsto a) \in C(\{x\}, \A)$. Now, for all  $n \in \N$ and $f \in C(\{x\}, \A_n)$, we have that $l^{\A_n}_{\mathsf{d}_X}(f)=0$ by definition. Thus, by construction of $\Lip^{\A, \beta}_{\mathsf{d}_X, \mathcal{I}, \tau}$ and $E^{\{x\}}_{\tau, n}$ for all $n \in \N$, the proof is complete. 

(6) This is immediate by construction, and in fact,  we have $\left(C(X, \A),\Lip^{\A, \beta}_{\mathsf{d}_X, \mathcal{I}, \tau} \right)=\left(C(X,\A), \Lip^{\A, r}_{\mathsf{d}_X, E^X_{\tau, 0}}\right)$.

Finally, we consider $C(X) \otimes \A$.  Since $\pi_{X, \A}$ is a *-isomorphism, we have that  $\left( C(X) \otimes \A, \Lip^{\A, \beta, \otimes}_{\mathsf{d}_X, \mathcal{I}, \tau}\right)$ and   $\left( C(X) \otimes \A_n, \Lip^{\A, \beta, \otimes}_{\mathsf{d}_X, \mathcal{I}, \tau}\right)$ for all $n \in \N$ are $2$-quasi-Leibniz compact quantum metric spaces, and by construction, we have 
$\left(C(X, \A),\Lip^{\A, \beta}_{\mathsf{d}_X, \mathcal{I}, \tau} \right)$ is fully quantum isometric to  $\left( C(X) \otimes \A, \Lip^{\A, \beta, \otimes}_{\mathsf{d}_X, \mathcal{I}, \tau}\right)$, and $\left(C(X, \A_n),\Lip^{\A, \beta}_{\mathsf{d}_X, \mathcal{I}, \tau} \right)$ is fully quantum isometric to  $\left( C(X) \otimes \A_n, \Lip^{\A, \beta, \otimes}_{\mathsf{d}_X, \mathcal{I}, \tau}\right)$ for all $n \in \N$ since the restriction of $\pi_{X, \A}$ to $C(X) \otimes \A_n$ is a *-isomorphism onto $C(X, \A_n)$. Thus, (1)-(6) (except for Expression \eqref{eq:af-h})   all hold with $C(X, \A)$ and $C(X, \A_n)$ replaced by $C(X)\otimes \A$ and $C(X) \otimes \A_n$, respectively, for all $n \in \N$, and  $\Lip^{\A, \beta}_{\mathsf{d}_X, \mathcal{I}, \tau}$ replaced by $\Lip^{\A, \beta, \otimes}_{\mathsf{d}_X, \mathcal{I}, \tau}$ since full quantum isometry is an equivalence relation and distance $0$ in $\qpropinquity{}$ is equivalent to existence of full quantum isometry by Theorem \ref{t:distq}.

Now, we verify Expression \eqref{eq:leibniz-tensor}.  Let $f \in C(X), a \in \A$. Let $f^X \in C(X, \A)$ denote the function defined for all $x \in X$ by $f^X(x)=f(x)\cdot 1_\A$.  Let $a^X \in C(X, \A)$ denote the function defined for all $x \in X$ by $a^X(x)=a$.  Thus $f^X \in C(X, \C1_\A)=C(X, \A_0)\subseteq C(X, \A_n)$ for all $n \in \N$. Also note that, we have that $\pi_{X,\A}(f \otimes a)=f^Xa^X.$  Hence, for all $n \in \N$, we have that $E^X_{\tau,n}(\pi_{X,\A}(f \otimes a))=E^X_{\tau,n}(f^Xa^X)=f^XE^X_{\tau,n}(a^X)$ for all $n \in \N$ since $E^X_{\tau,n}$ is a conditional expectation onto $C(X, \A_n)$ by Theorem \ref{t:ah-cond-exp}  and $f^X \in C(X, \A_0) \subseteq C(X, \A_n)$ and bimodule property of conditional expectations \cite[Definition 1.5.9 and Tomiyama Theorem  1.5.10]{Brown-Ozawa}. Furthermore, note that 
\begin{equation}\label{eq:cond-exp-tensor}
E^X_{\tau,n}(a^X)(x)=E_{\tau,n}(a^X(x))=E_{\tau,n}(a)
\end{equation} for all $x \in X$. Hence, we gather that for all $x,y\in X$,
\begin{align*}
\frac{\|\pi_{X,\A}(f \otimes a)(x)-\pi_{X,\A}(f \otimes a)(y)\|_\A}{\mathsf{d}_X(x,y)}& = \frac{\|f(x)\cdot  a-f(y) \cdot a\|_\A}{\mathsf{d}_X(x,y)}\\
& = \frac{|f(x)-f(y)|\cdot\|  a\|_\A}{\mathsf{d}_X(x,y)}\\
& =  \frac{|f(x)-f(y)|}{\mathsf{d}_X(x,y)}\cdot\|  a\|_\A
\end{align*}
and thus $
l^\A_{\mathsf{d}_X}(\pi_{X,\A}(f \otimes a))=l^\C_{\mathsf{d}_X}\|a\|_\A.$

Also, we gather that for all $n \in \N$
\begin{align*}
\frac{\|\pi_{X,\A}(f \otimes a)-E^X_{\tau,n}(\pi_{X,\A}(f \otimes a))\|_{C(X, \A)}}{\beta(n)}& = \frac{\|f^Xa^X-f^XE^X_{\tau,n}(a^X)\|_{C(X, \A)}}{\beta(n)}\\ 
& = \frac{\|f^X\cdot (a^X-E^X_{\tau,n}(a^X))\|_{C(X, \A)}}{\beta(n)}.
\end{align*}
Now $(f^X(a^X-E^X_{\tau,n}(a^X)))(x)=f(x)\cdot(a-E_{\tau,n}(a))=\pi_{X,\A}(f\otimes (a-E_{\tau,n}(a)))(x)$ for all $x \in X$ by Expression \eqref{eq:cond-exp-tensor}.  Thus $\pi_{X,\A}(f\otimes (a-E_{\tau,n}(a)))=f^X(a^X-E^X_{\tau,n}(a^X)).$  Hence, since $\pi_{X, \A}$ is a *-isomorphism and $\|\cdot\|_{C(X)\otimes \A}$ is a cross norm by \cite[Lemma 3.4.10]{Brown-Ozawa}, we have
\begin{align*}
\frac{\|\pi_{X,\A}(f \otimes a)-E^X_{\tau,n}(\pi_{X,\A}(f \otimes a))\|_{C(X, \A)}}{\beta(n)}& =  \frac{\|\pi_{X,\A}(f\otimes (a-E_{\tau,n}(a)))\|_{C(X, \A)}}{\beta(n)}\\
& =  \frac{\| f\otimes (a-E_{\tau,n}(a))\|_{C(X)\otimes \A}}{\beta(n)}\\
& = \frac{\| f\|_{C(X)}\cdot \|a-E_{\tau,n}(a)\|_\A}{\beta(n)}\\
& =  \frac{\|a-E_{\tau,n}(a)\|_\A}{\beta(n)}\cdot \| f\|_{C(X)},
\end{align*}
which proves Expression \eqref{eq:leibniz-tensor}.  The remaining equalities follow immediately from Expression \eqref{eq:leibniz-tensor}   the fact that $l^\C_{\mathsf{d}_X}(1_{C(X)})=0$  and $\Lip^\beta_{\mathcal{I}, \tau}(1_\A)=0$ by definition of a compact quantum metric space.  
\end{proof}
\begin{remark}
We note another advantage to Expression \eqref{eq:leibniz-tensor} aside from being related to the classical case in Expression \eqref{eq:leibniz-tensor-comm}.  If a seminorm $\Lip$ on a tensor product $\A \otimes \B$ satisfied Expression \eqref{eq:leibniz-tensor} with Lip-norms $\Lip_\A$ on $\A$ and $\Lip_\B$ on $\B,$ then $\Lip(a \otimes b)=0 \implies a\otimes b \in \C1_{\A \otimes \B}$.  Indeed, assume $a \otimes  b \neq0 \implies a \neq 0$ and $b \neq 0$. If  $\Lip(a \otimes b)=0 $, then $\Lip_\A(a)\|b\|_\B=0 \implies \Lip_\A(a)=0 \implies a \in \C 1_\A$ and $\Lip_\B(b)\|a\|_\A=0 \implies \Lip_\B(b)=0 \implies b \in \C 1_\B$, which together imply that $a \otimes b \in \C1_{\A \otimes \B}.$  

This condition (Expression \eqref{eq:leibniz-tensor}) on elementary tensors does not guarantee that $\Lip$ vanishes only on scalars with respect to all of $\A \otimes \B$. However, this observation along with Expression \eqref{eq:leibniz-tensor-comm} still suggests that a seminorm should satisfy  something like Expression \eqref{eq:leibniz-tensor}  in order to be considered a Lip-norm on a tensor product that is compatible with the tensor structure. 
\end{remark}

\section{The diameter of the quantum metric on $C(X) \otimes \A$}
Now, in Theorem \ref{t:ah-mk},  we will show that we can still capture much of the structure of $C(X)$ in the state space of $C(X, \A)$ with $\A$ as a unital AF algebra using the Monge-Kantorovich metric. In particular, with our quantum metric, we still have an isometric copy of $X$ in the  state space of $C(X, \A)$ or $C(X) \otimes \A$ and we find that the diameter of our quantum metric  in $C(X) \otimes \A$ is bounded by the diameter of the Cartesian product of classical quantum metric on $C(X)$ and our quantum metric   on $\A$. This is motivated by and reflects some properties of  the classical *-isomorphism between $C(X) \otimes C(Y)$ and $C(X \times Y)$ given in Theorem \ref{t:comm-tensor}.  This is seen more explicitly in Corollary \ref{c:comm-tensor-diam}. To understand the consequence of Theorem \ref{t:ah-mk}, we present a classical result (phrased in terms of quantum metrics) and its proof as it provides some details for the proof of Theorem \ref{t:ah-mk}.
\begin{theorem}\label{t:comm-mk-diam}
If $(X, \mathsf{d}_X)$ is a compact metric space, then using notation from Definition \ref{d:h-cqms}, we have
\[
\mathrm{diam}\left(\StateSpace(C(X)),\Kantorovich{l^\C_{\mathsf{d}_X}}\right)= \mathrm{diam}(X,\mathsf{d}_X),
\]
and moreover, for all $x,y \in X$, we have  
\[
\mathsf{d}_X(x,y)=\Kantorovich{l^\C_{\mathsf{d}_X}}(\delta_x, \delta_y),
\]
where $\delta_x: f \in C(X)\mapsto f(x) \in \C$ and $\delta_y: f \in C(X) \mapsto f(y) \in \C$ are the Dirac point masses associated to $x,y \in X$, respectively.
\end{theorem}
\begin{proof}
Let $\mu, \nu$ be two pure states on $C(X)$. Thus, there exists $x,y \in X$ such that $\mu=\delta_x: f \in C(X)\mapsto f(x) \in \C$ and $\nu=\delta_y: f \in C(X)\mapsto f(y) \in \C$  by  \cite[Thereom VII.8.7]{Conway90} and \cite[Theorem 5.1.6]{Murphy90}.   Now, it is classical result that for all $x,y \in X$, it holds that 
\[
\Kantorovich{l^\C_{\mathsf{d}_X}}(\delta_x,\delta_y)=\mathsf{d}_X(x,y)
\]
for which a proof is provided here \cite[Theorem 2.2.10]{Aguilar-Thesis}.  Hence, 
\[
\Kantorovich{l^\C_{\mathsf{d}_X}}(\delta_x,\delta_y)\leq \mathrm{diam}(X,\mathsf{d}_X).
\]
and thus $\Kantorovich{l^\C_{\mathsf{d}_X}}(\mu,\nu)\leq \mathrm{diam}(X,\mathsf{d}_X)$ for all pure states $\mu,\nu$ of $C(X)$ by the beginning of the proof.  Since $\Kantorovich{l^\C_{\mathsf{d}_X}}$ is  convex   in the sense of \cite[Definition 9.1]{Rieffel99} and the state space is the weak* closed convex hull of the pure states by \cite[Corollary 5.1.10]{Murphy90}, we have that 
\[
\Kantorovich{l^\C_{\mathsf{d}_X}}(\mu,\mu)\leq \mathrm{diam}(X,\mathsf{d}_X).
\] for all states $\mu,\nu$ on $C(X,\A)$ since  $\Kantorovich{l^\C_{\mathsf{d}_X}}$ metrizes the weak* topology by \cite{Kantorovich40} for which a proof in terms of quantum metrics can be found here \cite[Theorem 2.2.10]{Aguilar-Thesis}. Hence, 
\[
\mathrm{diam}\left(\StateSpace(C(X)),\Kantorovich{l^\C_{\mathsf{d}_X}}\right)\leq \mathrm{diam}(X,\mathsf{d}_X).
\]
Now, since $(X, \mathsf{d}_X)$ is compact, there exist $x_1, x_2$ such that $\mathsf{d}_X(x_1,x_2)=\mathrm{diam}(X,\mathsf{d}_X).$ Therefore, 
\begin{align*}
\mathrm{diam}(X,\mathsf{d}_X)&=\mathsf{d}_X(x_1,x_2)=\Kantorovich{l^\C_{\mathsf{d}_X}}(\delta_x,\delta_y) \leq \mathrm{diam}\left(\StateSpace(C(X)),\Kantorovich{l^\C_{\mathsf{d}_X}}\right)\\
& \leq \mathrm{diam}(X,\mathsf{d}_X),
\end{align*}
which completes the proof. 
\end{proof} 
Now, will show how our new quantum metrics extend the results of the above theorem when we tensor $C(X)$ by a unital AF algebra $\A$ equipped with a faithful tracial state, which provides another justification of our construction of Lip-norms.
\begin{theorem}\label{t:ah-mk}
Let $(X, \mathsf{d}_X)$ be a compact metric space and let $\A=\overline{\cup_{n \in \N} \A_n}^{\|\cdot\|_\A}$ be a unital C*-algebra equipped with faithful tracial state $\tau$ such that $\A_n$ is a finite-dimensional   C*-subalgebra of $\A$ for all $n \in \N$ and $\A_0=\C1_\A \subseteq \A_1 \subseteq \A_2 \subseteq \cdots $.  In particular, $\A$ is AF. Denote $\mathcal{I}=(\A_n)_{n \in \N}$. Let $(\beta(n))_{n \in \N}$ be a sequence of positive real numbers that converges to $0$. 

If $ \Lip^{\A, \beta, \otimes}_{\mathsf{d}_X, \mathcal{I}, \tau}$ is the Lip-norm on $C(X)\otimes \A$ from Theorem \ref{t:ah-c*-cqms}, then  using Notation \ref{n:prod-metric}, we have
\begin{equation}
\begin{split}\label{eq:tensor-diam}
 \mathrm{diam}\left(\StateSpace(C(X)\otimes \A),\Kantorovich{\Lip^{\A, \beta, \otimes}_{\mathsf{d}_X, \mathcal{I}, \tau}}\right) 
 & \leq \mathrm{diam}\left(\StateSpace(C(X))\times \StateSpace(\A),\mathsf{d}^1_{\Kantorovich{l^\C_{\mathsf{d}_X}}\times \Kantorovich{\Lip^{\beta}_{\mathcal{I}, \tau}}}\right)\\
&=  \mathrm{diam}\left(\StateSpace(C(X)) , \Kantorovich{l^\C_{\mathsf{d}_X}} \right)\\
& \quad + \mathrm{diam}\left(  \StateSpace(\A),  \Kantorovich{\Lip^{\beta}_{\mathcal{I}, \tau}}\right)\\
& = \mathrm{diam}(X, \mathsf{d}_X)+ \mathrm{diam}\left(  \StateSpace(\A),  \Kantorovich{\Lip^{\beta}_{\mathcal{I}, \tau}}\right)\\
& \leq \mathrm{diam}(X, \mathsf{d}_X)+2\beta(0),
\end{split}
\end{equation} 
where $\Lip^{\beta}_{\mathcal{I}, \tau}$ is from Theorem \ref{AF-lip-norms-thm-union}, and note that  $\left(\StateSpace(C(X)\otimes \A),\Kantorovich{\Lip^{\A, \beta, \otimes}_{\mathsf{d}_X, \mathcal{I}, \tau}}\right)$ can be replaced with $\left(  \StateSpace(C(X,\A)),  \Kantorovich{\Lip^{\A,\beta}_{\mathsf{d}_X,\mathcal{I}, \tau}}\right)$ in Expression \eqref{eq:tensor-diam}.

Furthermore, 
  for all states $\psi \in \StateSpace(  \A)$ (including $\tau$) and for all  $x,y \in X$, it holds that
\begin{equation}\label{eq:state-metric}
\mathsf{d}_X(x,y)=\Kantorovich{\Lip^{\A, \beta}_{\mathsf{d}_X, \mathcal{I}, \tau}}(\psi_x, \psi_y)
\end{equation}
where $\psi_x, \psi_y$ are the   states on $C(X, \A)$ defined in (3) of Theorem \ref{t:ah-cond-exp}, which induces isometries from \[(X, \mathsf{d}_X) \text{  into } \left(\StateSpace(C(X)\otimes \A),\Kantorovich{\Lip^{\A, \beta, \otimes}_{\mathsf{d}_X, \mathcal{I}, \tau}}\right)\] and from \[(X, \mathsf{d}_X) \text{  into } \left(  \StateSpace(C(X,\A)),  \Kantorovich{\Lip^{\A,\beta}_{\mathsf{d}_X,\mathcal{I}, \tau}}\right).\]
\end{theorem}
\begin{proof}
Expression \eqref{eq:tensor-diam}  follows some of the proof of \cite[Proposition 2.9]{Aguilar-Bice17}, but follows a different approach in the beginning to establish the relationship with the diameter on the quantum metric on the AF algebra.  Let $\mu, \nu$ be pure states on $C(X,\A)$.  By \cite[Lemma 2.8]{Aguilar-Bice17}, there exist $x,y\in X$ and pure states $\phi, \psi$ on $\A$ such that $\mu=\phi_x$ and $\nu=\psi_y$, where $\phi_x(a)=\phi(a(x))$ for all $a \in C(X, \A)$ and similarly for $\psi_y, \psi_y$.   Fix $a \in \sa{C(X, \A)}$ such that  $ \Lip^{\A, \beta}_{\mathsf{d}_X, \mathcal{I}, \tau}(a)\leq 1$.  Next, note that $a(x) \in \sa{\A}$ and   
\[
\|a(x)- E_{\tau,n}(a(x))\|_\A= \|a(x)- E^X_{\tau,n}(a)(x)\|_\A \leq \|a-E^X_{\tau,n}(a)\|_{C(X, \A)}
\]
for all $n \in \N$.  Therefore $\Lip^{\beta}_{\mathcal{I}, \tau}(a(x)) \leq 1$. Hence, we have 
\begin{equation}\label{quo-lip-eq}
\begin{split}
|\phi_x(a)-\psi_x(a) |&=|\phi(a(x))-\psi(a(x))| \leq \Kantorovich{\Lip^{\beta}_{\mathcal{I}, \tau}}(\phi,\psi) \leq  \mathrm{diam}\left(  \StateSpace(\A),  \Kantorovich{\Lip^{\beta}_{\mathcal{I}, \tau}}\right).
\end{split}
\end{equation}

 Furthermore, since $\Lip^{\A, \beta}_{\mathsf{d}_X, \mathcal{I}, \tau}(a)\leq 1$, we gather for all $n \in \N$
\begin{align*}
|\psi_x(a)-\psi_y(a) |& =|\psi(a(x))-\psi(a(y))| = |\psi(a(x)-a(y))|\leq \|a(x)-a(y)\|_\A \\
&    \leq  \mathsf{d}_X(x,y) \leq  \diam{X}{\mathsf{d}_X}. 
\end{align*}
 Combining with Expression \eqref{quo-lip-eq}, we have
\begin{align*}
|\mu(a)-\nu(a)|& = |\phi_x(a) - \psi_y(a)|  \leq |\phi_x(a) -\psi_x(a)|+|\psi_x(a)-\psi_y(a)|\\
&\leq \mathrm{diam}\left(  \StateSpace(\A),  \Kantorovich{\Lip^{\beta}_{\mathcal{I}, \tau}}\right)+ \diam{X}{\mathsf{d}_X}.
\end{align*}

Therefore, by definition, we have $\Kantorovich{\Lip^{\A, \beta}_{\mathsf{d}_X, \mathcal{I}, \tau}}(\mu,\nu) \leq  \mathrm{diam}\left(  \StateSpace(\A),  \Kantorovich{\Lip^{\beta}_{\mathcal{I}, \tau}}\right)+ \diam{X}{\mathsf{d}_X}$ for all pure states $\mu, \nu$ on $C(X, \A)$.  Next, since $\Kantorovich{\Lip^{\A, \beta}_{\mathsf{d}_X, \mathcal{I}, \tau}}$ is  convex   in the sense of \cite[Definition 9.1]{Rieffel99} and the state space is the weak* closed convex hull of the pure states by \cite[Corollary 5.1.10]{Murphy90} and $\Kantorovich{\Lip^{\A,\beta}_{\mathsf{d}_X,\mathcal{I}, \tau}}$ metrizes the weak* topology by Definition \ref{Monge-Kantorovich-def} and Theorem \ref{t:ah-c*-cqms}, we have that
\[
\mathrm{diam}\left(  \StateSpace(C(X,\A)),  \Kantorovich{\Lip^{\A,\beta}_{\mathsf{d}_X,\mathcal{I}, \tau}}\right)\leq  \mathrm{diam}\left(  \StateSpace(\A),  \Kantorovich{\Lip^{\beta}_{\mathcal{I}, \tau}}\right)+ \diam{X}{\mathsf{d}_X}
\]
  as $\mu,\nu$ were arbitrary pure states on $C(X,\A)$.   
  
  Now  by Theorem \ref{t:ah-c*-cqms},
  \[
  \mathrm{diam}\left(  \StateSpace(C(X)\otimes \A),  \Kantorovich{\Lip^{\A,\beta, \otimes}_{\mathsf{d}_X,\mathcal{I}, \tau}}\right)=\mathrm{diam}\left(  \StateSpace(C(X,\A)),  \Kantorovich{\Lip^{\A,\beta}_{\mathsf{d}_X,\mathcal{I}, \tau}}\right)
  \]
  since a full quantum isometry induces an isometry between the state spaces with their associated Monge-Kantorovich metrics by \cite[Theorem 6.2]{Rieffel00}. Next,  
  \[
  \mathrm{diam}\left(  \StateSpace(\A),  \Kantorovich{\Lip^{\beta}_{\mathcal{I}, \tau}}\right)\leq 2 \beta(0)
  \]
  by \cite[Corollary 3.10]{Aguilar-Latremoliere15}. Also, $\mathrm{diam}\left(\StateSpace(C(X)) , \Kantorovich{l^\C_{\mathsf{d}_X}} \right)  = \mathrm{diam}(X, \mathsf{d}_X)$ by Theorem \ref{t:comm-mk-diam}.  The rest follows from the   fact that the diameter of the product of compact metric spaces with the $1$-metric is the sum of the diameters of each compact metric space.

Next, we establish Expression \eqref{eq:state-metric}. This follows the proof of \cite[Theorem 3.7]{Aguilar-Bice17}, but we provide a simpler proof with a more general conclusion due the fact that in this article we only consider the C*-norm, whereas \cite{Aguilar-Bice17} considers other norms besides the C*-norm since it deals with matrix algebras. Let $x,y \in X$.  Define
\[
Y_{\mathsf{d}_X} : z \in X \mapsto \mathsf{d}_X(y,z)1_\A \in \A.
\]
Note $Y_{\mathsf{d}_X} \in \sa{C(X,\C1_\A)}=\sa{C(X,\A_0)} \subseteq \sa{C(X,\A_n)} \subseteq   \sa{C(X,\A)}$ for all $n \in \N$.  Hence, for all $n \in \N$, we have
\[
\|Y_{\mathsf{d}_X}-E^X_{\tau,n}(Y_{\mathsf{d}_X})\|_{C(X, \A)}= \|Y_{\mathsf{d}_X}- Y_{\mathsf{d}_X}\|_{C(X, \A)}=0
\]
by Theorem \ref{t:ah-cond-exp}. 
Also, for any $v,w \in X$, we have
\begin{align*}
\frac{\|Y_{\mathsf{d}_X}(v)-Y_{\mathsf{d}_X}(w)\|_\A}{\mathsf{d}_X(v,w)}& = \frac{\|\mathsf{d}_X(y,v)1_\A-\mathsf{d}_X(y,w)1_\A\|_\A}{\mathsf{d}_X(v,w)}\\
& =  \frac{|\mathsf{d}_X(y,v)-\mathsf{d}_X(y,w)|}{\mathsf{d}_X(v,w)}  \leq \frac{\mathsf{d}_X(v,w)}{\mathsf{d}_X(v,w)}=1.
\end{align*}
Hence $\Lip^{\A, \beta}_{\mathsf{d}_X, \mathcal{I}, \tau}(Y_{\mathsf{d}_X})\leq 1.$ Next, we have
\begin{align*}
|\psi_x(Y_{\mathsf{d}_X})-\psi_y(Y_{\mathsf{d}_X})|& = |\psi(Y_{\mathsf{d}_X}(x))-\psi(Y_{\mathsf{d}_X}(y))|\\
& = |\psi( \mathsf{d}_X(y,x)1_\A)-\psi( \mathsf{d}_X(y,y)1_\A)| = |\psi( \mathsf{d}_X(y,x)1_\A)|&\\
  = \mathsf{d}_X(y,x).
\end{align*}
Therefore
\[
\Kantorovich{\Lip^{\A, \beta}_{\mathsf{d}_X, \mathcal{I}, \tau}}(\psi_x, \psi_y)\geq \mathsf{d}_X(x,y).
\]
Now, let $a \in \sa{C(X, \A)}$ such that  $\Lip^{\A, \beta}_{\mathsf{d}_X, \mathcal{I}, \tau}(a)\leq 1.$ Then
\begin{align*}
|\psi_x(a)-\psi_y(a)|& = |\psi(a(x))-\psi(a(y))|  = |\psi( a(x)-a(y))| \\
&\leq \|a(x)-a(y)\|_\A  \leq \mathsf{d}_X(y,x).
\end{align*}
Hence
\[
\Kantorovich{\Lip^{\A, \beta}_{\mathsf{d}_X, \mathcal{I}, \tau}}(\psi_x, \psi_y)\leq \mathsf{d}_X(x,y),
\]
which implies that $\Kantorovich{\Lip^{\A, \beta}_{\mathsf{d}_X, \mathcal{I}, \tau}}(\psi_x, \psi_y)=\mathsf{d}_X(x,y)$.

Thus, the map
\[
\Psi: x \in (X, \mathsf{d}_X) \mapsto \psi_x \in \left(  \StateSpace(C(X,\A)),  \Kantorovich{\Lip^{\A,\beta}_{\mathsf{d}_X,\mathcal{I}, \tau}}\right)
\]
is an isometry into $\left(  \StateSpace(C(X,\A)),  \Kantorovich{\Lip^{\A,\beta}_{\mathsf{d}_X,\mathcal{I}, \tau}}\right)$. Since  $\left(  \StateSpace(C(X)\otimes \A),  \Kantorovich{\Lip^{\A,\beta, \otimes}_{\mathsf{d}_X,\mathcal{I}, \tau}}\right)$ is isometric onto  $\left(  \StateSpace(C(X,\A)),  \Kantorovich{\Lip^{\A,\beta}_{\mathsf{d}_X,\mathcal{I}, \tau}}\right)$ by \cite[Theorem 6.2]{Rieffel00} as mentioned above, the proof is complete.
\end{proof}

Now, we show what the above result translates to in the commutative case of AF algebras, which displays a satisfying relationship between the quantum metric and the classical metric structure. 

\begin{corollary}\label{c:comm-tensor-diam}
Let $(X, \mathsf{d}_X)$ be a compact metric space and let $(Y, \mathsf{d}_Y)$ be a totally disconnected compact metric space that contains more than one point (for instance, the Cantor space $\mathcal{C}$), and thus $\mathrm{diam}(Y, \mathsf{d}_Y)>0$. Since $C(Y)$ is AF by \cite[Proposition 3.1]{Bratteli74}, let $\mathcal{I}=(\A_n)_{n \in \N}$ be a non-decreasing sequence of finite-dimensional C*-subalgebras of $C(Y)$ such that $\A_0=\C1_{C(Y)}$. Let $(\beta(n))_{n \in \N}$ be a sequence of positive real numbers converging to $0$.

If we set $\beta(0)=\frac{\mathrm{diam}(Y, \mathsf{d}_Y)}{2}$, then using the *-isomorphism $c_T: C(X)\otimes C(Y)\rightarrow C(X\times Y)$ of Theorem \ref{t:comm-tensor}, then using Notation \ref{n:prod-metric}
\begin{align*}
\mathrm{diam}\left( \StateSpace(C(X \times Y)), \Kantorovich{\Lip^{\A,\beta, \otimes}_{\mathsf{d}_X, \mathcal{I},\tau}\circ c_T}\right)& \leq \mathrm{diam}\left( X \times Y, \mathsf{d}^1_{\mathsf{d}_X \times \mathsf{d}_Y}\right)\\
& = \mathrm{diam}\left( \StateSpace(C(X \times Y)), \Kantorovich{l^\C_{\mathsf{d}^1_{\mathsf{d}_X \times \mathsf{d}_Y}}}\right).
\end{align*}
We note that in the case that $Y$ is a one point, we have equality in the above, which is immediate by (4) of Theorem \ref{t:ah-c*-cqms} since $C(Y)\cong \C$ in this case.
\end{corollary}
\begin{proof}
Since $c_T$  is a *-isomorphism,  $\left(C(X \times Y), \Lip^{\A,\beta, \otimes}_{\mathsf{d}_X, \mathcal{I},\tau}\circ c_T \right)$ is a $2$-quasi-Leibniz compact quantum metric space. 
By construction, $c_T$ is a full quantum isometry from $\left( C(X) \otimes C(Y), \Lip^{\A,\beta, \otimes}_{\mathsf{d}_X, \mathcal{I},\tau}\right)$ onto $\left(C(X \times Y), \Lip^{\A,\beta, \otimes}_{\mathsf{d}_X, \mathcal{I},\tau}\circ c_T \right) $.  Hence 
\[\mathrm{diam}\left( \StateSpace(C(X \times Y)), \Kantorovich{\Lip^{\A,\beta, \otimes}_{\mathsf{d}_X, \mathcal{I},\tau}\circ c_T}\right)=\mathrm{diam}\left( \StateSpace(C(X) \otimes C(Y)), \Kantorovich{\Lip^{\A,\beta, \otimes}_{\mathsf{d}_X, \mathcal{I},\tau} }\right)\]
since a full quantum isometry induces an isometry between the state spaces with their associated Monge-Kantorovich metrics by \cite[Theorem 6.2]{Rieffel00}.

Therefore, by  Theorem \ref{t:ah-mk}, we have that
\begin{align*}
\mathrm{diam}\left( \StateSpace(C(X \times Y)), \Kantorovich{\Lip^{\A,\beta, \otimes}_{\mathsf{d}_X, \mathcal{I},\tau}\circ c_T}\right)&\leq \mathrm{diam}\left(X, \mathsf{d}_X\right)+2\beta(0)\\
& = \mathrm{diam}\left(X, \mathsf{d}_X\right)+\mathrm{diam}\left(Y, \mathsf{d}_Y\right)\\
& = \mathrm{diam}\left( X \times Y, \mathsf{d}^1_{\mathsf{d}_X\times \mathsf{d}_Y}\right)
\end{align*}
since the diameter of the product of compact metric spaces with respect to the $1$-metric is equal to the the sum of the diameters.  The final equality is provided by Theorem \ref{t:comm-mk-diam}.
\end{proof}
Thus, although the above result may not produce equality outside the classical case like in Theorem \ref{t:comm-mk-diam}, we still achieve an upper bound using the classical metric structure. Furthermore, given the Cantor space $\mathcal{C}$, we see that Corollary \ref{c:comm-tensor-diam} places two quantum metrics on $C(X \times \mathcal{C})$ given a compact metric space $(X, \mathsf{d}_X)$, where one quantum metric comes from the AF structure of $C(\mathcal{C})$ and the other comes from the metric structure of $\mathcal{C}$. It would be interesting to see how these two quantum metrics compare. Indeed, this is motivated our results with \Latremoliere{} in \cite{Aguilar-Latremoliere15} and L\'opez in \cite{Aguilar-Lopez19} where we compared  certain  quantum metrics on $C(\mathcal{C})$, in which these quantum metrics on $C(\mathcal{C})$  can be given by the case $C(\{x\}\times \mathcal{C})$ of Corollary \ref{c:comm-tensor-diam}. Hence, a study of $C(X \times \mathcal{C})$ should extend our results with \Latremoliere{} in \cite{Aguilar-Latremoliere15} and L\'opez in \cite{Aguilar-Lopez19} in a satisfying manner.

\section{Continuous families  of $C(X) \otimes \A$}

 In \cite{Aguilar-Latremoliere15}, we showed that UHF algebras  of Glimm \cite{Glimm60} vary continuously in Gromov-Hausdorff propinquity with respect to their multiplicity sequences in the Baire metric space. Now, we will show that this convergence result is not disrupted by tensoring UHF algebras by $C(X)$. We will now define the  Baire metric space, which a classical space that is vital to the study of Descriptive Set Theory, and is often called the {\em irrationals} since it is homeomorphic to the irrationals in $(0,1)$  \cite{Miller95}. For our purposes, it provides the ideal domain for our continuity results. Here is the definition of the Baire metric space.
 
 \begin{definition}[{\cite{Miller95}}]\label{d:Baire}
 Let $\mathcal{N}=(\N\setminus \{0\})^\N$. For each $x=(x(n))_{n \in \N}, y=( y(n))_{n \in \N}$ set
\[
\mathsf{d}_\mathcal{N}(x,y)=\begin{cases}
0&: \ x=y\\
2^{-\min\{m \in \N : x(m)\neq y(m)\}} & : \text{ otherwise.}
\end{cases}
\]
The metric space $(\mathcal{N}, \mathsf{d}_\mathcal{N})$ is called the {\em Baire space}.
 \end{definition}
 Since we are dealing with more structure than just the UHF algebra itself, we need to carefully define what we mean by UHF algebras since our Lip-norms required particular inductive sequences to be able to be defined. We note that we consider the entire class of UHF algebras up to *-isomorphism in the following notation by \cite{Bratteli72} since we capture all Bratteli diagrams of UHF algebras.
 \begin{notation}\label{n:uhf}
 Let $\beta \in \mathcal{N}$, we define the sequence $\boxtimes \beta \in \mathcal{N}$ by 
 \[
  \boxtimes \beta(n)=
 \begin{cases}
1& \text{: if } n=0\\
 \prod_{j=0}^{n-1}(\beta(j)+1)& \text{: otherwise.} 
 \end{cases}
 \]
 For each $n \in \N$, let $\alpha_{\beta,n}: M_{\boxtimes \beta(n)}(\C) \rightarrow M_{\boxtimes \beta(n+1)}(\C) $ be the unital *-monomorphism given for all $a \in M_{\boxtimes \beta(n)}(\C) $ by 
 \[
 \alpha_{\beta,n}( a)= \left( \begin{array}{lll}
 a & & 0\\
 & \ddots & \\
 0 & & a
 \end{array}\right), 
 \]
 where there are $\beta(n)+1$ copies of $a$ on the diagonal. Set $\boxtimes(\beta)^{-1}=(\boxtimes(\beta)^{-1})_{n \in \N}$.  Denote the inductive limit 
\[\alg{uhf}(\beta)=\underrightarrow{\lim}(M_{\boxtimes \beta(n)}(\C), \alpha_{\beta,n})_{n\in \N}\] of \cite[Section 6.1]{Murphy90}, and set 
 \[
 \mathcal{I}(\beta)=\left(\alpha^{(n)}_\beta(M_{\boxtimes \beta(n)}(\C))\right)_{n \in \N},
 \] 
 where for each $n \in \N$, $\alpha^{(n)}_\beta: M_{\boxtimes \beta(n)}(\C)\rightarrow\alg{uhf}(\beta)$ is the canonical unital *-monomorphism of \cite[Section 6.1]{Murphy90} such that $\alpha^{(n+1)}_\beta\circ \alpha_{\beta,n}=\alpha^{(n)}_\beta$, where the subalgebra  $\cup_{n \in \N}\alpha^{(n)}_\beta(M_{\boxtimes \beta(n)}(\C))$ is dense in $\alg{uhf}(\beta)$ and $\alpha^{(0)}_\beta(M_{\boxtimes \beta(0)}(\C))=\C1_{\alg{uhf}(\beta)}$.
 
 Let $\tau(\beta)$ denote the unique faithful tracial state on $\alg{uhf}(\beta)$ given by \cite[Example 6.2.1 and Remark 6.2.4]{Murphy90}.
 \end{notation}
 
 Now, we are ready to establish our main convergence result.
 
 \begin{theorem}\label{t:uhf-tensor-cont}
 Let $QCQMS_2$ denote the class of $2$-quasi-Leibniz compact quantum metric spaces.  Let $(X, \mathsf{d}_X)$ be a compact metric space.
 
 Using notation from Notation \ref{n:uhf} and  Theorem \ref{t:ah-c*-cqms},  the map
 \[
 \alg{uhf}^\otimes : \beta \in (\mathcal{N}, \mathsf{d}_\mathcal{N}) \mapsto \left(C(X)\otimes \alg{uhf}(\beta),\Lip^{\alg{uhf}(\beta),\boxtimes(\beta)^{-1} , \otimes}_{\mathsf{d}_X, \mathcal{I}(\beta),\tau(\beta)}\right) \in ( QCQMS_2, \qpropinquity{})
 \]
 is $2$-Lipschitz and thus continuous.
 
  The result is the same with the   space $\left(C(X)\otimes \alg{uhf}(\beta),\Lip^{\alg{uhf}(\beta),\boxtimes(\beta)^{-1} , \otimes}_{\mathsf{d}_X, \mathcal{I}(\beta),\tau(\beta)}\right) $ replaced with $\left(C(X,\alg{uhf}(\beta)),\Lip^{\alg{uhf}(\beta),\boxtimes(\beta)^{-1} }_{\mathsf{d}_X, \mathcal{I}(\beta),\tau(\beta)}\right) $
 \end{theorem}
 \begin{proof}
  The map is well-defined by Theorem \ref{t:ah-c*-cqms}, so let $\eta, \beta \in \mathcal{N}$ such that $\eta\neq\beta$.  Note that $\boxtimes \eta(n)\geq 2^n$ and $\boxtimes \beta(n)\geq 2^n$ for all $n \in \N$. Thus, there exists $N \in \N$ such that $\mathsf{d}_\mathcal{N}(\eta, \beta)=2^{-N}$, and thus the first coordinate $\eta$ and $\beta$ disagree is $N$, and thus $\boxtimes\beta(n)=\boxtimes\eta(n)$ for all $n \leq N$. Our estimate will be the same for $N=0$ and $N=1$ since each inductive sequence $\mathcal{I}(\eta)$ and $\mathcal{I}(\beta)$ begins with the scalars. So, we assume $N \geq 1$. 
  
  Next, by Theorem \ref{t:ah-c*-cqms}, we have
  \begin{align*}
   \qpropinquity{}\left( \alg{uhf}^\otimes(\eta), 
   \left(C(X)\otimes \alpha^{(N)}_\eta(M_{\boxtimes \eta(N)}(\C)),\Lip^{\alg{uhf}(\eta),\boxtimes(\eta)^{-1} , \otimes}_{\mathsf{d}_X, \mathcal{I}(\eta),\tau(\eta)}\right) \right) 
  &  \leq \boxtimes(\eta)(N)^{-1}\\
  &\leq 2^{-N}= \mathsf{d}_{\mathcal{N}}(\eta, \beta)
  \end{align*}
  and 
   \begin{align*}
   \qpropinquity{}\left( \alg{uhf}^\otimes(\beta),   \left(C(X)\otimes \alpha^{(N)}_\beta(M_{\boxtimes \beta(N)}(\C)),\Lip^{\alg{uhf}(\beta),\boxtimes(\beta)^{-1} , \otimes}_{\mathsf{d}_X, \mathcal{I}(\beta),\tau(\beta)}\right) \right) 
 &  \leq \boxtimes(\beta)(N)^{-1}\\
 &\leq 2^{-N}= \mathsf{d}_{\mathcal{N}}(\eta, \beta).
  \end{align*}
  
  Next, we show  that  \[\left(C(X)\otimes \alpha^{(N)}_\eta(M_{\boxtimes \eta(N)}(\C)),\Lip^{\alg{uhf}(\eta),\boxtimes(\eta)^{-1} , \otimes}_{\mathsf{d}_X, \mathcal{I}(\eta),\tau(\eta)}\right) \] and \[\left(C(X)\otimes \alpha^{(N)}_\beta(M_{\boxtimes \beta(N)}(\C)),\Lip^{\alg{uhf}(\beta),\boxtimes(\beta)^{-1} , \otimes}_{\mathsf{d}_X, \mathcal{I}(\beta),\tau(\beta)}\right)\] are fully quantum isometric.
  
   Set $N'=\boxtimes\beta(N)=\boxtimes\eta(N)$. We will show that \[\left(C(X, \alpha^{(N)}_\eta(M_{N'}(\C))),\Lip^{\alg{uhf}(\eta),\boxtimes(\eta)^{-1}}_{\mathsf{d}_X, \mathcal{I}(\eta),\tau(\eta)}\right) \] and \[\left(C(X, \alpha^{(N)}_\beta(M_{N'}(\C))),\Lip^{\alg{uhf}(\beta),\boxtimes(\beta)^{-1} }_{\mathsf{d}_X, \mathcal{I}(\beta),\tau(\beta)}\right)\] are fully quantum isometric, which is equivalent to the previous statement  by   Theorem \ref{t:ah-c*-cqms}.  Consider the *-isomorphism 
   \[ F=(\alpha^{(N)}_\eta)\circ (\alpha^{(N)}_\beta)^{-1}: \alpha^{(N)}_\beta(M_{N'}(\C))\rightarrow \alpha^{(N)}_\eta(M_{N'}(\C)).\]
   Now, define
   \[
   F^X: a \in C(X, \alpha^{(N)}_\beta(M_{N'}(\C))) \longmapsto (x \mapsto F(a(x))) \in C(X, \alpha^{(N)}_\eta(M_{N'}(\C))),
   \]
   which is a *-isomorphism by construction. We will show that $F^X$ is a full quantum isometry.

   Let $a \in C(X, \alpha^{(N)}_\beta(M_{N'}(\C)))$.  Thus, for each $x \in X$ there exists a unique $a_x \in M_{N'}(\C) $ such that $a(x)=\alpha^{(N)}_\beta(a_x)$.  Now, let $x,y \in X$, we have
   \begin{align*}
   \|F^X(a)(x)-F^X(a)(y)\|_{\alg{uhf}(\eta)}& = \|F(a(x))-F(a(y))\|_{\alg{uhf}(\eta)}\\
   & = \|F(\alpha^{(N)}_\beta(a_x))-F(\alpha^{(N)}_\beta(a_y))\|_{\alg{uhf}(\eta)}\\
   & = \| \alpha^{(N)}_\eta(a_x)- \alpha^{(N)}_\eta(a_y)\|_{\alg{uhf}(\eta)}\\ 
   & = \|a_x-a_y\|_{M_{N'}(\C)} \\
   &= \|\alpha^{(N)}_\beta(a_x)-\alpha^{(N)}_\beta(a_y)\|_{M_{N'}(\C)}  = \|a(x)-a(y)\|_{\alg{uhf}(\beta)}.
   \end{align*}
   Thus, $l^{\alg{uhf}(\beta)}_{\mathsf{d}_X}(a)=l^{\alg{uhf}(\eta)}_{\mathsf{d}_X}(F^X(a)).$ 
   
   Next, fix $x \in X$. Then let $n \in \{0, \ldots, N-1\}$, we have
  \begin{align*}
  &\|F^X(a)(x)-E^X_{\tau(\eta),n}(F^X(a))(x)\|_{\alg{uhf}(\eta)}\\
  & =   \|F^X(a)(x)-E_{\tau(\eta),n}(F^X(a)(x))\|_{\alg{uhf}(\eta)}\\
  & =  \|F(\alpha^{(N)}_\beta(a_x))-E_{\tau(\eta),n}(F(\alpha^{(N)}_\beta(a_x)))\|_{\alg{uhf}(\eta)}\\
  & =  \| \alpha^{(N)}_\eta(a_x)-E_{\tau(\eta),n}(\alpha^{(N)}_\eta(a_x))\|_{\alg{uhf}(\eta)}\\
  & =  \| \alpha^{(N)}_\beta(a_x)-E_{\tau(\beta),n}(\alpha^{(N)}_\beta(a_x))\|_{\alg{uhf}(\beta)} =  \| a(x)-E_{\tau(\beta),n}(a(x))\|_{\alg{uhf}(\beta)},
  \end{align*}
  where the second to the last equality is provided by \cite[Expression (4.4)]{Aguilar-Latremoliere15} and is due to the uniqueness of faithful tracial state on matrix algebras and the fact that the conditional expectation is the orthogonal projection onto  matrix algebra constructed by the inner product induced by the faithful tracial state. Hence, taking supremums over $x \in X$, we have
  \[
  \frac{\|F^X(a)-E^X_{\tau(\eta),n}(F^X(a))\|_{C(X, \alg{uhf}(\eta))}}{\boxtimes \eta(n)}=\frac{\| a-E^X_{\tau(\beta),n}( a)\|_{C(X, \alg{uhf}(\beta))}}{\boxtimes \beta(n)}
  \]
  since $\boxtimes \beta(n)=\boxtimes \eta(n)$ for all $n \leq N$. Therefore, we have that 
  \[
\Lip^{\alg{uhf}(\eta),\boxtimes(\eta)^{-1} }_{\mathsf{d}_X, \mathcal{I}(\eta),\tau(\eta)}\circ F^X(a)=  \Lip^{\alg{uhf}(\beta),\boxtimes(\beta)^{-1} }_{\mathsf{d}_X, \mathcal{I}(\beta),\tau(\beta)}(a)
  \]
for all  $a \in C(X, \alpha^{(N)}_\beta(M_{N'}(\C)))$. Thus $F^X$ is a full quantum isometry and therefore
\begin{align*}
& \qpropinquity{}\bigg( \left(C(X, \alpha^{(N)}_\eta(M_{N'}(\C))),\Lip^{\alg{uhf}(\eta),\boxtimes(\eta)^{-1}}_{\mathsf{d}_X, \mathcal{I}(\eta),\tau(\eta)}\right), \\
& \quad \quad \quad \left(C(X, \alpha^{(N)}_\beta(M_{N'}(\C))),\Lip^{\alg{uhf}(\beta),\boxtimes(\beta)^{-1} }_{\mathsf{d}_X, \mathcal{I}(\beta),\tau(\beta)}\right)\bigg)=0
\end{align*}
by Theorem \ref{t:distq}.

Combining this with the beginning of the proof and the triangle inequality, we have
\begin{align*}
 \qpropinquity{}\left(\alg{uhf}^\otimes(\eta) ,  
\alg{uhf}^\otimes(\beta)\right) 
& \leq \mathsf{d}(\beta, \eta)+0+\mathsf{d}(\beta, \eta)=2\mathsf{d}(\beta, \eta),
\end{align*}
which shows $2$-Lipschitz and thus continuity.  The last statement is provided by the full quantum isometry between the quantum metric spaces from the proof of Theorem \ref{t:ah-c*-cqms}. 
 \end{proof}
 
Lastly, we show that we can approximate $C(X, \A)$ with finite-dimensional C*-algebras in propinquity even though these spaces need not be AF. We accomplish by using finite-dimensional approximations for $C(X)$ in propinquity and by varying the Lip-norm on $ \A$, so we achieve some form of convergence on a product.  We also provide estimates between $C(X, \A)$ and $C(Y, \A)$ depending on the metric geometry of $X$ and $Y$  and the quantum metric on $\A$.
\begin{theorem}\label{t:fd-approx}
Let $(X, \mathsf{d}_X)$ and $(Y, \mathsf{d}_Y)$ be compact metric spaces. Let $\A=\overline{\cup_{n \in \N} \A_n}^{\|\cdot\|_\A}$ be a unital C*-algebra equipped with faithful tracial state $\tau$ such that $\A_n$ is a unital C*-subalgebra of $\A$ for all $n \in \N$ and $\A_0=\C1_\A \subseteq \A_1 \subseteq \A_2 \subseteq \cdots $.  In particular, $\A$ is AF. Denote $\mathcal{I}=(\A_n)_{n \in \N}$.

If $(\beta(n))_{n \in \N}$ be a sequence of positive real numbers that converge to $0$, then using notation from Theorem \ref{t:ah-c*-cqms}, we have
\begin{equation}\label{eq:fd-approx}
\qpropinquity{}\left( \left(C(X) \otimes \A, \Lip^{\A, \beta, \otimes}_{\mathsf{d}_X, \mathcal{I}, \tau}\right), (C(Y), l^\C_{\mathsf{d}_Y})\right) \leq \beta(0)+\mathrm{GH}(X,Y),
\end{equation}
where $\mathrm{GH}$ is the Gromov-Hausdorff distance between compact metric spaces \cite{burago01}, and furthermore, we have 
\begin{align*}
\qpropinquity{}\left( \left(C(X) \otimes \A, \Lip^{\A, \beta, \otimes}_{\mathsf{d}_X, \mathcal{I}, \tau}\right), \left(C(Y) \otimes \A, \Lip^{\A, \beta, \otimes}_{\mathsf{d}_Y, \mathcal{I}, \tau}\right)\right) \leq 2\beta(0)+\mathrm{GH}(X,Y).
\end{align*}

Moreover, if we let $\varepsilon>0$ and choose $\beta(0)< \varepsilon/2$, then there exists a finite $X_\varepsilon \subseteq X$ such that 
\begin{align*}
\qpropinquity{}\left( \left(C(X) \otimes \A, \Lip^{\A, \beta, \otimes}_{\mathsf{d}_X, \mathcal{I}, \tau}\right), (C(X_\varepsilon), l^\C_{\mathsf{d}_X})\right) < \varepsilon,
\end{align*}
where $\dim(C(X_\varepsilon))< \infty$.

The above results all hold with $ \left(C(X) \otimes \A, \Lip^{\A, \beta, \otimes}_{\mathsf{d}_X, \mathcal{I}, \tau}\right)$ replaced with  $\left(C(X, \A), \Lip^{\A, \beta}_{\mathsf{d}_X, \mathcal{I}, \tau}\right).$
\end{theorem}
\begin{proof}
We note that $\left(C(X,\C1_\A) , \Lip^{\A, \beta, \otimes}_{\mathsf{d}_X, \mathcal{I}, \tau}\right)$ is fully quantum isometric to $(C(X), l^\C_{\mathsf{d}_X})$ by the same argument of \cite[Corollary 2.12]{Aguilar-Bice17}.  Thus, by the triangle inequality, Theorem \ref{t:distq}, and Theorem \ref{t:ah-c*-cqms}, we have
\begin{align*}
& \qpropinquity{}\left( \left(C(X) \otimes \A, \Lip^{\A, \beta, \otimes}_{\mathsf{d}_X, \mathcal{I}, \tau}\right), (C(Y), l^\C_{\mathsf{d}_Y})\right)\\
 & \leq \qpropinquity{}\left( \left(C(X) \otimes \A, \Lip^{\A, \beta, \otimes}_{\mathsf{d}_X, \mathcal{I}, \tau}\right), \left(C(X,\C1_\A) , \Lip^{\A, \beta, \otimes}_{\mathsf{d}_X, \mathcal{I}, \tau}\right)\right) \\
 & \quad + \qpropinquity{}\left(\left(C(X,\C1_\A) , \Lip^{\A, \beta, \otimes}_{\mathsf{d}_X, \mathcal{I}, \tau}\right),  (C(Y), l^\C_{\mathsf{d}_Y})\right)\\
 & \leq \beta(0)+ \qpropinquity{}\left((C(X), l^\C_{\mathsf{d}_X}),  (C(Y), l^\C_{\mathsf{d}_Y})\right)\\
 & \leq \beta(0)+\mathrm{GH}(X,Y),
\end{align*} 
where the last inequality is given by \cite[Theorem  6.6]{Latremoliere13}. Similarly, we have
\begin{align*}
& \qpropinquity{}\left( \left(C(X) \otimes \A, \Lip^{\A, \beta, \otimes}_{\mathsf{d}_X, \mathcal{I}, \tau}\right), \left(C(Y) \otimes \A, \Lip^{\A, \beta, \otimes}_{\mathsf{d}_Y, \mathcal{I}, \tau}\right)\right)\\
& \leq\qpropinquity{}\left( \left(C(X) \otimes \A, \Lip^{\A, \beta, \otimes}_{\mathsf{d}_X, \mathcal{I}, \tau}\right), \left(C(X,\C1_\A) , \Lip^{\A, \beta, \otimes}_{\mathsf{d}_X, \mathcal{I}, \tau}\right)\right) \\
 & \quad + \qpropinquity{}\left(\left(C(X,\C1_\A) , \Lip^{\A, \beta, \otimes}_{\mathsf{d}_X, \mathcal{I}, \tau}\right),  \left(C(Y,\C1_\A) , \Lip^{\A, \beta, \otimes}_{\mathsf{d}_Y, \mathcal{I}, \tau}\right)\right)\\
 & \quad + \qpropinquity{}\left( \left(C(Y,\C1_\A) , \Lip^{\A, \beta, \otimes}_{\mathsf{d}_Y, \mathcal{I}, \tau}\right),\left(C(Y) \otimes \A, \Lip^{\A, \beta, \otimes}_{\mathsf{d}_Y, \mathcal{I}, \tau}\right)\right) \\
 & \leq \beta(0)+\mathrm{GH}(X,Y)+\beta(0).
\end{align*}

Now, since $(X, \mathsf{d}_X)$ is a compact metric space, there exists a finite $\varepsilon/2$-net $X_\varepsilon\subseteq X$ of $X$.  Hence $\mathrm{GH}(X,X_\varepsilon)\leq \varepsilon/2,$ since the Hausdorff distance on compact subsets of a compact metric space dominates the Gromov-Hausdorff distance, which shows 
\[\qpropinquity{}\left( \left(C(X) \otimes \A, \Lip^{\A, \beta, \otimes}_{\mathsf{d}_X, \mathcal{I}, \tau}\right), (C(X_\varepsilon), l^\C_{\mathsf{d}_X})\right) < \varepsilon/2+\varepsilon/2\]
by Expression \ref{eq:fd-approx}.  The proof is complete  up to the last sentence of the theorem, which   is provided by the full quantum isometry between the quantum metric spaces from the proof of Theorem \ref{t:ah-c*-cqms}. 
\end{proof}

\bibliographystyle{plain}
\bibliography{thesis-a}

\begin{thebibliography}{10}

\bibitem{Aguilar16}
{K}. {A}guilar.
\newblock Convergence of quotients of {A}{F} algebras in quantum propinquity by
  convergence of ideals.
\newblock 48 pages, (Accepted 2019), to appear in {\em Journal of Operator
  Theory}, ArXiv: 1608.07016.

\bibitem{Aguilar18}
{K}. {A}guilar.
\newblock Inductive limits of {C}*-algebras and compact quantum metric spaces.
\newblock 24 pages, submitted (2018), ArXiv: 1807.10424.

\bibitem{Aguilar-Thesis}
{K}. {A}guilar.
\newblock {\em Quantum {M}etrics on {A}pproximately {F}inite-{D}imensional
  {A}lgebras}.
\newblock ProQuest LLC, Ann Arbor, MI, 2017.
\newblock Thesis (Ph.D.)--University of Denver.

\bibitem{Aguilar-Bice17}
{K}. {A}guilar and {T}. {B}ice.
\newblock Standard homogeneous {C}*-algebras as compact quantum metric spaces.
\newblock 32 pages, (Accepted 2018) to appear in {\em Banach Center
  Publications}, ArXiv: 1711.08846.

\bibitem{Aguilar-Latremoliere15}
{K}. {A}guilar and {F}. {L}atr{\'e}moli{\`e}re.
\newblock Quantum ultrametrics on {A}{F} algebras and the {G}romov-{H}ausdorff
  propinquity.
\newblock {\em Studia Mathematica}, 231(2):149 --193, 2015.
\newblock ArXiv: 1511.07114.

\bibitem{Aguilar-Lopez19}
{K}. {A}guilar and {A}. {L}\'opez.
\newblock A quantum metric on the {C}antor {S}pace.
\newblock 22 pages, submitted (2019), ArXiv: 1907.05835.

\bibitem{Blackadar06}
B.~Blackadar.
\newblock {\em Operator algebras}, volume 122 of {\em Encyclopaedia of
  Mathematical Sciences}.
\newblock Springer-Verlag, Berlin, 2006.
\newblock Theory of $C^*$-algebras and von Neumann algebras, Operator Algebras
  and Non-commutative Geometry, III.

\bibitem{Bourbaki-top2}
{N}. {B}ourbaki.
\newblock {\em General topology. {C}hapters 5--10}.
\newblock Elements of Mathematics (Berlin). Springer-Verlag, Berlin, 1998.
\newblock Translated from the French, Reprint of the 1989 English translation.

\bibitem{Bratteli72}
{O}. {B}ratteli.
\newblock Inductive limits of finite dimensional {$C^\ast$-algebras}.
\newblock {\em Trans. Amer. Math. Soc.}, 171:195--234, 1972.

\bibitem{Bratteli74}
Ola Bratteli.
\newblock Structure spaces of approximately finite-dimensional {$C\sp{\ast}
  $}-algebras.
\newblock {\em J. Functional Analysis}, 16:192--204, 1974.

\bibitem{Brown-Ozawa}
{N}.~{P}. {B}rown and {N}. {O}zawa.
\newblock {\em C*-Algebras and Finite-Dimensional Approximations}, volume~88 of
  {\em Graudate Studies in Mathematics}.
\newblock American Mathematical Society, 2008.

\bibitem{burago01}
{D}. {B}urago, {Y}. {B}urago, and {S}. {I}vanov.
\newblock {\em A course in Metric Geometry}, volume~33 of {\em Graduate Texts
  in Mathematics}.
\newblock American Mathematical Society, 2001.

\bibitem{Connes89}
A.~{C}onnes.
\newblock Compact metric spaces, {F}redholm modules and hyperfiniteness.
\newblock {\em Ergodic Theory and Dynamical Systems}, 9(2):207--220, 1989.

\bibitem{Connes}
A.~{C}onnes.
\newblock {\em Noncommutative Geometry}.
\newblock Academic Press, San Diego, 1994.

\bibitem{Conway90}
{J}~{B}. {C}onway.
\newblock {\em A Course in Functional Analysis}, volume~96 of {\em Graduate
  Texts in Mathematics}.
\newblock Springer-Verlag, 1990.

\bibitem{Davidson}
{K}.~{R}. {D}avidson.
\newblock {\em {C*}--Algebras by Example}.
\newblock Fields Institute Monographs. American Mathematical Society, 1996.

\bibitem{Glimm60}
{J}. {G}limm.
\newblock On a certain class of operator algebras.
\newblock {\em Trans. Amer. Math. Soc.}, 95:318--340, 1960.

\bibitem{Hausdorff}
{F}. {H}ausdorff.
\newblock {\em {G}rundz{\"u}ge der {M}engenlehre}.
\newblock Verlag Von Veit und Comp., 1914.

\bibitem{Kaad19}
{J}. {K}aad and {D}. {K}yed.
\newblock Dynamics of compact quantum metric spaces.
\newblock 42 pages (2019), ArXiV: 1904.13278.

\bibitem{Kadison97}
{R}. {K}adison and {J}. {R}ingrose.
\newblock {\em Fundamentals of the Theory of Operator Algebras {I}}, volume~15
  of {\em Graduate Studies in Mathematics}.
\newblock AMS, 1997.

\bibitem{Kantorovich40}
{L}.~{V}. {K}antorovich.
\newblock On one effective method of solving certain classes of extremal
  problems.
\newblock {\em Dokl. Akad. Nauk. USSR}, 28:212--215, 1940.

\bibitem{Kaplansky51}
{I}. {K}aplansky.
\newblock The structure of certain operator algebras.
\newblock {\em Trans. Amer. Math. Soc.}, 70:219--255, 1951.

\bibitem{Kerr02}
D.~{K}err.
\newblock Matricial quantum {G}romov-{H}ausdorff distance.
\newblock {\em J. Funct. Anal.}, 205(1):132--167, 2003.
\newblock math.OA/0207282.

\bibitem{Kerr09}
{D}. {K}err and {H}. {L}i.
\newblock On {G}romov--{H}ausdorff convergence of operator metric spaces.
\newblock {\em J. Oper. Theory}, 1(1):83--109, 2009.

\bibitem{Latremoliere05}
{F}. {L}atr{\'e}moli{\`e}re.
\newblock Approximation of the quantum tori by finite quantum tori for the
  quantum {G}romov-{H}ausdorff distance.
\newblock {\em Journal of Funct. Anal.}, 223:365--395, 2005.
\newblock math.OA/0310214.

\bibitem{Latremoliere13c}
{F}. {L}atr{\'e}moli{\`e}re.
\newblock Convergence of fuzzy tori and quantum tori for the quantum
  {G}romov-{H}ausdorff propinquity: an explicit approach.
\newblock {\em M\"unster J. Math.}, 8(1), 2015.
\newblock ArXiv: math/1312.0069.

\bibitem{Latremoliere13b}
{F}. {L}atr{\'e}moli{\`e}re.
\newblock The dual {G}romov--{H}ausdorff {P}ropinquity.
\newblock {\em Journal de Math{\'e}matiques Pures et Appliqu{\'e}es},
  103(2):303--351, 2015.
\newblock ArXiv: 1311.0104.

\bibitem{Latremoliere16}
{F}. {L}atr{\'e}moli{\`e}re.
\newblock The modular {G}romov-{H}ausdorff propinquity.
\newblock {\em Submitted}, 2016.
\newblock Submitted (2016), 67 pages, ArXiv: 1608.04881.

\bibitem{Latremoliere15}
{F}. {L}atr{\'e}moli{\`e}re.
\newblock A compactness theorem for the dual {G}romov-{H}ausdorff propinquity.
\newblock {\em Indiana Univ. Math. J.}, 66(5):1707--1753, 2017.

\bibitem{Latremoliere18a}
{F}. {L}atr{\'e}moli{\`e}re.
\newblock The dual modular {G}romov-{H}ausdorff propinquity and completeness.
\newblock {\em Submitted}, 2018.
\newblock 36 pages (2018), ArXiv: 1811.04534.

\bibitem{Latremoliere18c}
{F}. {L}atr{\'e}moli{\`e}re.
\newblock The {G}romov-{H}ausdorff propinquity for metric spectral triples.
\newblock {\em Submitted}, 2018.
\newblock 28 pages (2018), ArXiv: 1811.10843.

\bibitem{Latremoliere13}
{F}. {L}atr{\'e}moli{\`e}re.
\newblock The quantum {G}romov-{H}ausdorff {P}ropinquity.
\newblock {\em Trans. Amer. Math. Soc.}, pages : 49 Pages,
  http://dx.doi.org/10.1090/tran/6334, to appear in print, electronically
  published on May 22, 2015.
\newblock ArXiv: 1302.4058.

\bibitem{Miller95}
{A}.~{W}. {Miller}.
\newblock {\em Descriptive Set Theory and Forcing: how to Prove Theorems about
  Borel Sets the hard way}.
\newblock Springer-Verlag, 1995.

\bibitem{Murphy90}
{G}.~{J}. {M}urphy.
\newblock {\em ${C^\ast}$-algebras and Operator theory}.
\newblock Academic Press, San Diego, 1990.

\bibitem{Ozawa05}
{N}. {O}zawa and M.~A. {R}ieffel.
\newblock Hyperbolic group {$C\sp\ast$}-algebras and free products
  {$C\sp\ast$}-algebras as compact quantum metric spaces.
\newblock {\em Canad. J. Math.}, 57:1056--1079, 2005.
\newblock ArXiv: math/0302310.

\bibitem{Rieffel98a}
M.~A. {R}ieffel.
\newblock Metrics on states from actions of compact groups.
\newblock {\em Documenta Mathematica}, 3:215--229, 1998.
\newblock math.OA/9807084.

\bibitem{Rieffel99}
M.~A. {R}ieffel.
\newblock Metrics on state spaces.
\newblock {\em Documenta Math.}, 4:559--600, 1999.
\newblock math.OA/9906151.

\bibitem{Rieffel05}
M.~A. {R}ieffel.
\newblock Compact quantum metric spaces.
\newblock In {\em Operator algebras, quantization, and noncommutative
  geometry}, volume 365 of {\em Contemporary Math}, pages 315--330. American
  Mathematical Society, 2005.
\newblock ArXiv: 0308207.

\bibitem{Rieffel15}
M.~A. {R}ieffel.
\newblock Matricial bridges for ``matrix algebras converge to the sphere''.
\newblock In {\em Operator algebras and their applications}, volume 671 of {\em
  Contemp. Math.}, pages 209--233. Amer. Math. Soc., Providence, RI, 2016.
\newblock ArXiv: 1502.00329.

\bibitem{Rieffel18}
M.~A. {R}ieffel.
\newblock Vector bundles for ``matrix algebras converge to the sphere''.
\newblock {\em Journal of Geometry and Physics}, 132:181--204, 2018.
\newblock ArXiv: 1711.04054.

\bibitem{Rieffel00}
M.~A. {R}ieffel.
\newblock {G}romov-{H}ausdorff distance for quantum metric spaces.
\newblock {\em Mem. Amer. Math. Soc.}, 168(796), March 2004.
\newblock math.OA/0011063.

\bibitem{Willard}
{S}. {W}illard.
\newblock {\em General Topology}.
\newblock Dover Publications, Inc., 2004.

\end{thebibliography}

\end{document}